\documentclass[11pt]{article}
\input{style.sty}

\begin{document}
\title{Solving mathematical programs with complementarity constraints \\ by disjunctive regularizations}
\author{
S. L\"ammel
\thanks{
Department of Mathematics, Chemnitz University of Technology,
Reichenhainer Str. 41, 09126
Chemnitz, Germany; e-mail: sebastian.laemmel@mathematik.tu-chemnitz.de (corresponding author), vladimir.shikhman@mathematik.tu-chemnitz.de.
 } \and V. Shikhman\samethanks[1]
}

\maketitle

\maketitle
\vspace{-5ex}
\abstract{We propose a new disjunctive regularization for mathematical programs with complementarity constraints (MPCC). Its feasible set coincides with that of the Kanzow-Schwartz regularization. However, their functional descriptions differ considerably. For the disjunctive regularization, the logical operator OR and equivalent max-type constraints are used. Unlike the Kanzow-Schwartz, the disjunctive regularization satisfies the tailored linear independence constraint qualification if the original MPCC does. More than that, the favorable convergence properties -- known to hold for the Kanzow-Schwartz regularization -- remain valid for the disjunctive regularization as well. In particular, no second order necessary conditions are required to guarantee convergence towards S-stationary points of MPCC. Additionally, we keep track of the topological type of approximating and limiting nondegenerate C-stationary points in terms of their C-indices. Quadratic and biactive parts of the C-indices are shown to generically correspond to each other while regularizing. This is a new phenomenon as compared to the Scholtes or sign-type regularizations studied before. Numerical experiments illustrate that the proposed disjunctive regularization clearly outperforms the Kanzow-Schwartz regularization. 
Its numerical performance is even better than that of the Scholtes regularization if solving MPCCs with high accuracy.
}

\vspace{2ex}

\vspace{2ex}
{\bf MSC-classification: 90C33, 49M20
}

\section{Introduction}\label{sec:intro}

We consider mathematical programs with complementarity constraints:
\[
\mbox{MPCC}: \quad
\min_{x} \,\, f(x)\quad \mbox{s.\,t.} \quad x \in M
\]
with the feasible set
\[
M=\left\{x \in\R^n\, \left\vert\,
F_{1,j}(x) \cdot F_{2,j}(x)=0, F_{1,j}(x) \ge 0, F_{2,j}(x)\ge 0, j=1,\ldots,\kappa \right.\right\},
\]
where the defining functions $f \in C^2(\R^n,\R)$,
$F_1, F_2 \in C^2(\R^n,\R^\kappa)$ are twice continuously differentiable. A typical MPCC feasible set $M$ is depicted in Figure \ref{fig:feasiblesetMPCC}, where $n=2$, $\kappa=1$, $F_{1,1}(x)=x_1$ and $F_{2,1}(x)=x_2$. 
Complementarity constraints usually arise as Karush-Kuhn-Tucker conditions in the context of generalized Nash equilibrium problems, bilevel optimization, and semi-infinite programming, see e.g.~\cite{dempe:2002, stein:2003}. 

\begin{figure}[h]
	\centering
	\begin{tikzpicture}[xscale=0.55,yscale=0.55,domain=-0.5:5,samples=100]
	\draw node[below left] {0};
	\draw[->] (-0.5,0) -- (5,0);
	\draw[blue,line width=2] (0,0) -- (4.75,0) node[below,black] {$F_{1,j}(x)$};
	\draw[blue, line width=2] (0,0) -- (0,4.75) ;
	\draw[] (0,-0.5) -- (0,0);
	\draw[->] (0,4) -- (0,5) node[left] {$F_{2,j}(x)$};
	\end{tikzpicture}
    \begin{tikzpicture}[xscale=0.55,yscale=0.55,domain=-0.5:5,samples=100]
    \draw node[below left] {0};
    \fill [blue!20]
        (0, 0)
      -- plot[domain=0:1/4.75] (\x, {4.75})
      -- (1/4.75, 0);
    \fill [blue!20]
        (1/4.75, 0)
      -- plot[domain=1/4.75:4.75] (\x,{1/ \x})
      -- (4.75, 0);
    \draw[blue] plot[domain=1/4.75:4.75] (\x,{1/ \x});
    \draw[->] (-0.5,0) -- (5,0);
        \draw[->] (0,4) -- (0,5) node[left] {$F_{2,j}(x)$};
     \draw[blue,line width=2] (0,0) -- (4.75,0) node[below,black] {$F_{1,j}(x)$};
    \draw[blue, line width=2] (0,0) -- (0,4.75) ;
    \draw[] (0,-0.5) -- (0,0);
\end{tikzpicture}
\begin{tikzpicture}[xscale=0.55,yscale=0.55,domain=-0.5:5,samples=100]
    \draw node[below left] {0};
    \fill [blue!20]
        (0, 0)
      -- plot[domain=0:1] (\x, {4.75})
      -- (1,0);
    \fill [blue!20]
        (0.9, 0)
      -- plot[domain=0.9:4.75] (\x,{1})
      -- (4.75, 0);
    \draw[blue] (1,1) -- (4.75,1);
    \draw[blue] (1,1) -- (1,4.75);
    \draw[->] (-0.5,0) -- (5,0);
        \draw[->] (0,4) -- (0,5) node[left] {$F_{2,j}(x)$};
     \draw[blue,line width=2] (0,0) -- (4.75,0) node[below,black] {$F_{1,j}(x)$};
    \draw[blue, line width=2] (0,0) -- (0,4.75) ;
    \draw[] (0,-0.5) -- (0,0);
\end{tikzpicture}
	\caption{Feasible sets $M$, $M^{\text{S}(t)}$, and $M^{\text{KS}(t)}={M}^{\text{D}(t)}$}  
	\label{fig:feasiblesetMPCC}
\end{figure}
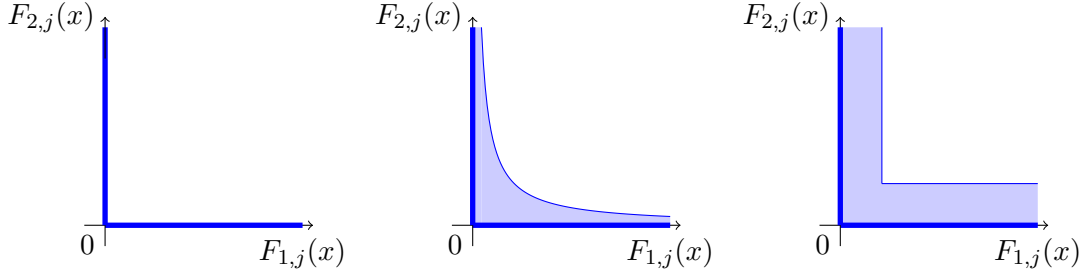

As for numerical methods to tackle MPCCs, various regularization schemes were proposed in literature so far, see e.g.~\cite{hoheisel:2013} for an overview. 
It all started with the Scholtes regularization from the seminal paper \cite{scholtes:2001}:
\[
\text{S}(t):\quad
\min_{x} \,\, f(x)\quad \mbox{s.\,t.} \quad x \in M^{\text{S}(t)}
\]
with the feasible set
\[
    M^{\text{S}(t)}=\left\{x \in\R^n\, \left\vert\,
    F_{1,j}(x) \cdot F_{2,j}(x)\le t, F_{1,j}(x) \ge 0,F_{2,j}(x)\ge 0, j=1,\ldots,\kappa
    \right.\right\},
\]
where the parameter $t>0$ is positive.
In Figure \ref{fig:feasiblesetMPCC}, the Scholtes regularization of the typical MPCC feasible set is depicted. Convergence results for the Scholtes regularization $\text{S}(t)$ focus on its Karush-Kuhn-Tucker points if $t \rightarrow 0$, see \cite{scholtes:2001}. Under the MPCC-tailored linear independence constraint qualification (MPCC-LICQ), they are shown to converge to a C-stationary point of MPCC. If additionally the multipliers for biactive complementarity constraints do not vanish (MPCC-ND2), one obtains an  M-stationary point in the limit. Assuming in addition a certain second order necessary condition (SONC) to hold there, even S-stationarity of the limiting point can be guaranteed. We mention that S-stationarity is necessary for optimality of MPCC under MPCC-LICQ. M-stationarity just adds singular saddle points of first order to the latter. The concept of C-stationarity is topologically motivated and includes also regular saddle points of first order, see \cite{shikhman:2025}. 

Another attempt to regularize MPCCs has been undertaken in \cite{kanzow:Relax}. There, the following Kanzow-Schwartz regularization has been proposed:
\[
\mbox{KS}(t): \quad
\min_{x} \,\, f(x)\quad \mbox{s.\,t.} \quad x \in M^{\text{KS}(t)}
\]
with the feasible set
\[
    M^{\text{KS}(t)}=\left\{
    x \in\R^n\, \left\vert\,   
        \Phi_j \left(x,t\right)\le 0,
    F_{1,j}(x) \ge 0, F_{2,j}(x)\ge 0, j=1,\ldots,\kappa 
            \right. \right\},
\]
where
    \[
\Phi_j \left(x,t\right)=\varphi(F_{1,j}(x)-t,F_{2,j}(x)-t), \quad
\varphi \left(a, b\right)=\left\{
\begin{array}{ll}
    a\cdot b & \mbox{for } a+b\ge 0, \\
     -\frac{1}{2}\left(a^2+b^2\right) & \mbox{for } a+b< 0.
\end{array}\right.
\]
We refer to Figure \ref{fig:feasiblesetMPCC}, where the Kanzow-Schwartz regularization of the typical MPCC feasible set is depicted. In \cite{kanzow:Relax}, it has been suggested to treat $\text{KS}(t)$ as an instance of nonlinear programming (NLP). Theoretically, the Kanzow-Schwartz regularization has better convergence properties compared to the Scholtes regularization. Under MPCC-LICQ, the Karush-Kuhn-Tucker points of $\text{KS}(t)$ converge towards an M- rather than C-stationary point of MPCC. If one additionally assumes MPCC-ND2 there, S- instead of M-stationarity of the limiting point follows.
Nevertheless, better theoretical properties of $\text{KS}(t)$ could not be practically confirmed, see \cite{hoheisel:2013, nurkanovic:2024}. According to these numerical studies, the Scholtes regularization seems to outperform the Kanzow-Schwartz regularization by accuracies in both function value and constraint violation, as well as by calculation time. From our viewpoint, this can be explained  by the fact that the Kanzow-Schwartz regularization $\text{KS}(t)$ happens to be degenerate as NLP even if the original MPCC is nondegenerate. E.g., the standard linear independence constraint qualification (LICQ) is potentially violated on its feasible set $M^{\text{KS}(t)}$. This degeneracy may lead to the catastrophic cancellation of multipliers if computing $\varepsilon$-stationary points of $\text{KS}(t)$ as it has been suggested to do in \cite{kanzow:Epsilon}. 
Another issue with the Kanzow-Schwartz regularization is that $M^{\text{KS}(t)}$ exhibits "re-entering" corners, cf.~Figure \ref{fig:feasiblesetMPCC}. This is typical for disjunctive optimization problems -- being treated first in \cite{jongen:1997} -- rather than for NLPs. Sadly enough, even viewed as an instance of disjunctive optimization the Kanzow-Schwartz regularization ${\text{KS}(t)}$ may violate the disjunctive-tailored linear independence constraints qualification (DISJ-LICQ).

Motivated by these shortcomings of $\text{KS}(t)$, we suggest a new disjunctive regularization of MPCC:
\[
\text{D}(t): \quad
\min_{x} \,\, f(x)\quad \mbox{s.\,t.} \quad x \in {M}^{\text{D}(t)}
\]
with the feasible set
\[
    {M}^{\text{D}(t)}=\left\{x \in\R^n\, \left\vert\,
    \max\{t-F_{1,j}(x), t-F_{2,j}(x)\}\ge 0, F_{1,j}(x) \ge 0, F_{2,j}(x)\ge 0, j=1,\ldots,\kappa \right.\right\}.
\]
 The term "disjunctive" is used here since the max-type constraint in the description of $M^{\text{D}(t)}$ can be equivalently written as 
\[
    t-F_{1,j}(x) \ge  0 \quad \mbox{or} \quad t-F_{2,j}(x) \ge 0.
\]
Note that the feasible sets of $\text{KS}(t)$ and $\text{D}(t)$ coincide, i.e.~$M^{\text{KS}(t)}=M^{\text{D}(t)}$, cf.~Figure \ref{fig:feasiblesetMPCC}. In turn, it is crucial that their functional descriptions differ.
The main advantage of considering $\text{D}(t)$ instead of $\text{KS}(t)$ is that DISJ-LICQ for the former is now inherited from the original MPCC-LICQ. Further, the favorable convergence properties over the Scholtes regularization $\text{S}(t)$ remain valid for $\text{D}(t)$. Under MPCC-LICQ, M-stationary points of $\text{D}(t)$ are shown to converge towards an M-stationary point of MPCC as $t \rightarrow 0$. Moreover, they approach even an S-stationary point, if MPCC-ND2 holds there. Exactly as in the case of Kanzow-Schwartz regularization, it is not necessary to assume SONC for the latter claim. The analogy to the theoretical results from \cite{kanzow:Relax} becomes complete, after we note that the S-stationary points of $\text{D}(t)$ turn out to be Karush-Kuhn-Tucker points of $\text{KS}(t)$ and vice versa.  

We enlarge our convergence analysis of $\text{D}(t)$ beyond what has been done for $\text{KS}(t)$ in \cite{kanzow:Relax}. Since DISJ-LICQ generically holds for $\text{D}(t)$, we may investigate if the topological type of its nondegenerate C-stationary points and their limiting counterparts for MPCC change. The topological type of a C-stationary point is captured by its C-index, which has a biactive and a quadratic part. The biactive part is related to the signs of multipliers for biactive constraints, whereas the quadratic part focuses on the number of negative eigenvalues of the restricted Hessians of the corresponding Lagrange function. By means of the C-index we may classify nondegenerate C-stationary points, in order to unambiguously distinguish minimizers and different kinds of saddle points from each other in algebraic terms. In Theorem \ref{thm:ttoc} on convergence and in Theorem \ref{thm:wellposedness} on well-posedness for $\text{D}(t)$, we successively keep track of the C-indices while regularizing. Generically, their quadratic and biactive parts with respect to the approaching and limiting C-stationary points of $\text{D}(t)$ and MPCC, respectively, are shown to correspond to each other. 
This is in strong contrast to the Scholtes regularization, for which the change of topological types has been fully clarified in \cite{laemmel:anomalies}. There, the only present quadratic index of Karush-Kuhn-Tucker points of $\text{S}(t)$ has been shown to nontrivially split into the biactive and quadratic parts of the C-index of the limiting C-stationary point of MPCC. Our investigation of indices explains why the convergence towards S-stationary points for the Scholtes regularization needs SONC, while for the disjunctive regularization it does not. 
Also for the sign-type regularization from \cite{kadrani:2009}, the behavior of the corresponding indices is different. There, the index shift stably occurs while regularizing and bifurcation phenomena prevail, see \cite{laemmel:sign}. 

As a consequence of our refined convergence analysis on the disjunctive regularization, the nondegenerate minimizers of MPCC with biactive
complementarity constraints will be necessarily approached by de facto Karush-Kuhn-Tucker points of $\text{D}(t)$. This generic avoiding of "re-entering" corners by approximating minimizers of the disjunctive regularization can explain its surprisingly good numerical performance. We test Kanzow-Schwartz, disjunctive, and Scholtes regularizations on small and middle-sized MPCC instances from the MacMPEC database of \cite{macmpec:2015}. As solvers for  $\text{KS}(t)$, $\text{D}(t)$, and $\text{S}(t)$, we use SLSQP from \cite{scipy:2020}, IPOPT from \cite{ipopt:2006}, and LogMIP from \cite{logmip}. Overall, the proposed disjunctive regularization clearly outperforms the Kanzow-Schwartz regularization. This concerns normalized relative error with respect to the objective function values, maximum constraint violation, and computational time. Surprisingly enough, the numerical performance of the disjunctive regularization is even better than that of the Scholtes regularization if solving MPCCs with high accuracy. As far as we know, this is for the first time in literature that a regularization numerically competitive to that of Scholtes is constructed. 

The organization of our paper is as follows. Section \ref{sec:preliminaries} is devoted to the basics on MPCC and disjunctive optimization. Main results on convergence and well-posedness properties of the disjunctive regularization $\text{D}(t)$ can be found in Section \ref{sec:main}. In Section \ref{sec:kanzow}, we relate $\text{D}(t)$ to the Kanzow-Schwartz  regularization $\text{KS}(t)$ and discuss the latter from the emerging perspective. Section \ref{sec:num} reports numerical results on $\text{D}(t)$ in comparison to the Scholtes and Kanzow-Schwartz regularizations.

Our notation is standard. We denote the $n$-dimensional Euclidean space by $\R^n$.
Given a twice continuously differentiable function $f:\R^n\to \R$, we denote its gradient by $\nabla f$ as a column vector and its Hessian matrix by $D^2f$. 


\section{Preliminaries}
\label{sec:preliminaries}


\subsection{Preliminaries on MPCC}\label{sec:pre-mpcc}


In what follows, we briefly recall the basics from the theory of MPCC, see e.g.~\cite{shikhman:2012}. 
To this end, the index sets of active constraints associated with $\bar x \in M$ will be helpful:
\[
\begin{array}{c}
a_{01}\left(\bar x\right)=\left\{j\,\left\vert\,  F_{1,j}(\bar x)=0, F_{2,j}(\bar 
x)>0\right.\right\},\\
a_{10}\left(\bar x\right)=\left\{j\,\left\vert\,  F_{1,j}(\bar x)>0, F_{2,j}(\bar  x)=0\right.\right\}, \\
a_{00}\left(\bar x\right)=\left\{j\,\left\vert\,  F_{1,j}(\bar x)=0, F_{2,j}(\bar  x)=0\right.\right\}.  
\end{array}
\]

We start our exposition with the MPCC-tailored linear independence constraint qualification.

\begin{definition}[MPCC-LICQ]
We say that a feasible point $\bar x \in M$ satisfies the MPCC-tailored linear independence constraint qualification (MPCC-LICQ) if the following vectors are linearly independent:
\[
\nabla F_{1,j}\left(\bar x\right), j \in a_{01}\left(\bar x\right)\cup a_{00}\left(\bar x\right),\quad
\nabla F_{2,j}\left(\bar x\right), j \in a_{10}\left(\bar x\right)\cup a_{00}\left(\bar x\right).
\]
\end{definition}
It is well-known that MPCC-LICQ is not restrictive in the sense of genericity.  
The subset of defining functions $F_1, F_2$, for which MPCC-LICQ is fulfilled at all
feasible points $\bar x \in M$ of a corresponding MPCC,  is $C^2_s$-open and -dense with respect to the strong (or Whitney-) topology, see \cite{scholtes:2001a}.

Next, we proceed with the different stationarity notions for MPCC used in the literature.

\begin{definition}[Stationarity for MPCC]
    \label{def:c-stat}
A feasible point $\bar x \in M$ is called C-stationary for MPCC if there exist multipliers
\[
\begin{array}{l}
\bar \sigma_{1,j}, j \in a_{01}\left(\bar x\right),\quad \bar \sigma_{2,j}, j \in a_{10}\left(\bar x\right),\quad \bar \varrho_{1,j},\bar \varrho_{2,j}, j \in a_{00}\left(\bar x\right),
\end{array}
\]such that 
\begin{equation}
   \label{eq:cstat-1} 
   \begin{array}{rcl}
   \displaystyle \nabla f(\bar x)&=& \displaystyle \sum\limits_{j \in a_{01}\left(\bar x\right)} \bar \sigma_{1,j}\nabla  F_{1,j}\left(\bar x\right)
    + \sum\limits_{j \in a_{10}\left(\bar x\right)} \bar \sigma_{2,j}\nabla  F_{2,j}\left(\bar x\right) \\
    &&+ \displaystyle\sum\limits_{j \in a_{00}\left(\bar x\right)} \left(\bar \varrho_{1,j} \nabla  F_{1,j}\left(\bar x\right)+\bar \varrho_{2,j}\nabla  F_{2,j}\left(\bar x\right)\right), \end{array}
\end{equation}
\begin{equation}
   \label{eq:cstat-2} \bar \varrho_{1,j} \cdot \bar \varrho_{2,j}\ge 0 \mbox{ for all } j \in a_{00}\left(\bar x\right).
\end{equation}
Moreover, the C-stationary point $\bar x$ is called:
\begin{itemize}
\item M-stationary if 
\begin{equation}
       \label{eq:mstat-2} \bar \varrho_{1,j} > 0, \bar \varrho_{2,j} > 0 \mbox{ or } \bar \varrho_{1,j} \cdot \bar \varrho_{2,j} = 0\mbox{ for all } j \in a_{00}\left(\bar x\right);
\end{equation}
    \item S-stationary if 
\begin{equation}
       \label{eq:sstat-2} \bar \varrho_{1,j} \ge 0, \bar \varrho_{2,j}\ge 0 \mbox{ for all } j \in a_{00}\left(\bar x\right).
\end{equation}
\end{itemize}
\end{definition}
It is clear that 
C-, M-, and S-stationarity become tighter in a row.
Moreover, S-stationarity turns out to be necessary for optimality. If $\bar x \in M$ is a local minimizer of MPCC satisfying MPCC-LICQ, then it is S-stationary, see \cite{scheel:2000}.
For a C-stationary point $\bar x \in M$ with multipliers $(\bar \sigma, \bar \varrho)$ -- which are unique under MPCC-LICQ -- it is convenient to define the appropriate Lagrange function as
\[
 \begin{array}{rcl}
\displaystyle L(x)&=&\displaystyle f(x)-\sum\limits_{j \in a_{01}\left(\bar x\right)} \bar \sigma_{1,j} F_{1,j}\left( x\right)
    - \sum\limits_{j \in a_{10}\left(\bar x\right)} \bar \sigma_{2,j} F_{2,j}\left( x\right) \\ &&- \displaystyle\sum\limits_{j \in a_{00}\left(\bar x\right)} \left(\bar \varrho_{1,j}  F_{1,j}\left(x\right)+\bar \varrho_{2,j} F_{2,j}\left(x\right)\right).
\end{array}
\]
%
%
The corresponding tangent space is given by
\[
\mathcal{T}_{\bar x}=\left\{\xi \in \R^n \,\left\vert \,
\begin{array}{l}
\nabla^T F_{1,j}\left(\bar x\right)\xi=0, j \in a_{01}\left(\bar x\right) \cup  a_{00}\left(\bar x\right), \\
\nabla^T F_{2,j}\left(\bar x\right)\xi=0, j \in a_{10}\left(\bar x\right) \cup  a_{00}\left(\bar x\right)
\end{array}
\right.\right\}.
\]

Now, we turn our attention to the nondegeneracy notion for C-stationary points.

\begin{definition}[Nondegenerate C-stationarity for MPCC]
    A C-stationary point $\bar x \in M$ of MPCC with multipliers $(\bar \sigma, \bar \varrho)$ is called nondegenerate if

MPCC-ND1: MPCC-LICQ holds at $\bar x$;

MPCC-ND2: the multipliers for biactive constraints do not vanish, i.e.~$\bar\varrho_{1,j}\cdot \bar\varrho_{2,j}> 0$,\,$j\in a_{00}\left(\bar x\right)$;
    
MPCC-ND3: the restricted Hessian matrix $D^2 L(\bar x)\restriction_{\mathcal{T}_{\bar x}}$ is nonsingular.

\noindent
For a nondegenerate C-stationary point we eventually use the following additional condition:

MPCC-ND4: $\bar \sigma_{1,j}\ne 0$ for all $j \in a_{01}\left(\bar x\right)$ and $\bar \sigma_{2,j}\ne 0$ for all $j\in a_{10}\left(\bar x\right)$.
\end{definition}
%
%
We note that nondegeneracy of C-stationary points is a generic property too.  
The subset of MPCC defining functions $f, F_1, F_2$, for which each
C-stationary point of a corresponding MPCC is nondegenerate, i.e.~fulfills MPCC-ND1, MPCC-ND2, and MPCC-ND3, is $C^2_s$-open and -dense, see \cite{jongen:2009}. Condition MPCC-ND4 has been introduced while studying stability of the Scholtes regularization for MPCC, see \cite{laemmel:anomalies}. There, MPCC-ND4 has been also shown to hold generically at all nondegenerate C-stationary points of MPCC.
For a nondegenerate C-stationary point its C-index is known to become a crucial invariant, cf. \cite{ralph:2011}. 

\begin{definition}[C-index for MPCC]
    Let $\bar x \in M$ be a nondegenerate C-stationary point of MPCC with unique multipliers $\left(\bar \sigma,\bar \varrho\right)$. The number of negative eigenvalues of the matrix $D^2 L(\bar x)\restriction_{\mathcal{T}_{\bar x}}$ is called its quadratic index ($\mbox{MPCC-}QI$). The number of negative pairs $\bar \varrho_{1,j}, \bar \varrho_{2,j}, j\in a_{00}\left(\bar x\right)$ with $\bar \varrho_{1,j}, \bar \varrho_{2,j} <0 $ is called the biactive index ($\mbox{MPCC-}BI$) of $\bar x$. We define the C-index  as the sum of both, i.e.~$\mbox{MPCC-}CI=\mbox{MPCC-}QI+\mbox{MPCC-}BI$.
\end{definition}
We emphasize that C-stationary points are topologically relevant in the sense of Morse theory, see \cite{jongen:2009}. It is to say that they adequately describe the topological changes of lower level sets of MPCC while their levels rise. First, outside the set of C-stationary points, the topology of the MPCC lower level sets remains unchanged, i.e.~the bigger level set is homeomorhic to the smaller one. Second, if passing a nondegenerate C-stationary point, a cell is attached along the boundary of the smaller lower level set to get the bigger one up to the homotopy equivalence. The dimension of the cell to be attached corresponds to the C-index of the nondegenerate C-stationary point under consideration. Additionally, the C-index also encodes  the local structure of MPCC in the vicinity of a nondegenerate C-stationary point. Nondegenerate C-stationary points with C-index equal to zero are local minimizers of MPCC. For nonvanishing C-indices we obtain all kinds of saddle points for MPCC. Overall, the C-index uniquely determines the topological type of a nondegenerate C-stationary point.

Finally, we define the following auxiliary index subsets, which depend on the signs of multipliers $(\bar \sigma, \bar \rho)$ corresponding to a C-stationary point $\bar x \in M$:
\[
\begin{array}{ll}
a_{01}^-\left(\bar x\right)=\left\{j \in a_{01}\left(\bar x\right)\,\left\vert\,  \bar \sigma_{1,j} <0 \right.\right\}, & a_{10}^-\left(\bar x\right)=\left\{j \in a_{10}\left(\bar x\right)\,\left\vert\,  \bar \sigma_{2,j} <0 \right.\right\}, \\
a_{01}^0\left(\bar x\right)=\left\{j \in a_{01}\left(\bar x\right)\,\left\vert\,  \bar \sigma_{1,j} =0 \right.\right\}, & a_{10}^0\left(\bar x\right)=\left\{j \in a_{10}\left(\bar x\right)\,\left\vert\,  \bar \sigma_{2,j} =0 \right.\right\},\\
a_{01}^+\left(\bar x\right)=\left\{j \in a_{01}\left(\bar x\right)\,\left\vert\,  \bar \sigma_{1,j} >0 \right.\right\}, & a_{10}^+\left(\bar x\right)=\left\{j \in a_{10}\left(\bar x\right)\,\left\vert\,  \bar \sigma_{2,j} >0 \right.\right\},\\
\end{array}
\]
\[
\begin{array}{l}
a_{00}^-\left(\bar x\right)=\left\{j \in a_{00}\left(\bar x\right)\,\left\vert\,  \bar \varrho_{1,j},\bar \varrho_{2,j} <0 \right.\right\},\\
a_{00}^0\left(\bar x\right)=\left\{j \in a_{00}\left(\bar x\right)\,\left\vert\,  \bar \varrho_{1,j} \cdot \bar \varrho_{2,j} =0 \right.\right\},
\\
a_{00}^+\left(\bar x\right)=\left\{j \in a_{00}\left(\bar x\right)\,\left\vert\,  \bar \varrho_{1,j},\bar \varrho_{2,j} >0 \right.\right\}.
\end{array}
\]


\subsection{Preliminaries on disjunctive optimization problems}\label{sec:pre-disj}

In what follows, we briefly apply the theory of disjunctive optimization from \cite{jongen:1997} to our regularization $\text{D}(t)$. To this end, the index sets of active constraints associated with $x \in M^{\text{D}(t)}$ will be helpful:

\[
\mathcal{H}_{12}(x)=\left\{j\,\left\vert \, F_{1,j}(x)= F_{2,j}(x) = t \right.\right\},
\]
\[
\mathcal{H}_{1}(x)=\left\{j\,\left\vert \, F_{1,j}(x)= t, F_{2,j}(x) > t \right.\right\}, \quad
\mathcal{H}_{2}(x)=\left\{j\,\left\vert \, F_{1,j}(x) > t, F_{2,j}(x) = t \right.\right\},
\]
\[
\mathcal{N}_1(x)=\left\{j\,\left\vert \, F_{1,j}(x)= 0 \right.\right\},\quad \mathcal{N}_2(x)=\left\{j\,\left\vert \, F_{2,j}(x)= 0 \right.\right\}.
\]

We start with the disjunctive-tailored linear independence constraint qualification.

\begin{definition} [DISJ-LICQ]
We say that a feasible point $x \in M^{\text{D}(t)}$ of $\text{D}(t)$ satisfies the linear independence constraint qualification tailored for disjunctive optimization (DISJ-LICQ) if the following vectors are linearly independent:
\[
\nabla F_{1,j}\left(x\right),\,j \in \mathcal{H}_{12}\left(x\right)\cup \mathcal{H}_{1}\left(x\right) \cup \mathcal{N}_1\left(x\right),\quad \nabla F_{2,j}\left(x\right), j \in \mathcal{H}_{12}\left(x\right)\cup \mathcal{H}_{2}\left(x\right) \cup \mathcal{N}_2\left(x\right).
\]    
\end{definition}

It turns out that DISJ-LICQ is inherited for the disjunctive regularization $\text{D}(t)$ if the original MPCC fulfills MPCC-LICQ. 

\begin{theorem}
    [MPCC- vs.~DISJ-LICQ]
\label{thm:LICQ}
Let a feasible point $\bar x \in M$ of MPCC fulfill MPCC-LICQ. Then, DISJ-LICQ holds at all feasible points $x \in M^{\text{D}(t)}$ of $\text{D}(t)$ for all sufficiently small $t$, whenever they are sufficiently close to $\bar x$. 
\end{theorem}
\begin{proof}
We show $a_{10}\left(\bar x\right) \cap \left(\mathcal{H}_{12}\xt\cup\mathcal{H}_{1}\xt  \cup \mathcal{N}_{1}\xt\right)=\emptyset$ for all $t$ sufficiently small.
In fact, suppose $j \in \mathcal{H}_{12}\xt\cup\mathcal{H}_{1}\xt  \cup\mathcal{N}_{1}\xt$ along some sequence of $x^t \in M^{\text{D}(t)}$ with $x^t \to \bar x$ for $t \to 0$. Then $F_{1,j}\xt\le t$ and thus converging to $0$ for $t \to 0$. Therefore, $j \notin a_{10}\left(\bar x\right)$ for all $t$ sufficiently small.
Hence, $\mathcal{H}_{12}\xt\cup\mathcal{H}_{1}\xt \cup \mathcal{N}_{1}\xt \subset a_{01}\left(\bar x\right) \cup a_{00}\left(\bar x\right)$. With similar arguments, we deduce $\mathcal{H}_{12}\xt\cup\mathcal{H}_{2}\xt  \cup \mathcal{N}_{2}\xt \subset a_{10}\left(\bar x\right) \cup a_{00}\left(\bar x\right)$. Therefore, MPCC-LICQ implies that the vectors
\[
\nabla F_{1,j}\left(\bar x\right), j \in \mathcal{H}_{12}\xt\cup\mathcal{H}_{1}\xt  \cup \mathcal{N}_{1}\xt, \quad
\nabla F_{2,j}\left(\bar x\right), j \in \mathcal{H}_{12}\xt\cup\mathcal{H}_{2}\xt  \cup \mathcal{N}_{2}\xt
\]
are linearly independent for all $t$ sufficiently small, and $x^t$ sufficiently close to $\bar x$. DISJ-LICQ follows directly from the continuity of $\nabla F_1$ and  $\nabla F_2$. 
    \qed
\end{proof}

Various stationarity notions for disjunctive optimization problems are given in the following definition. The topologically relevant notion of stationarity has been identified in \cite{jongen:1997}. We shall refer to it as C-stationarity due to its possible derivation by means of the Clarke's normal cone, cf. \cite{shikhman:2025}. M- and S-stationarity has been defined for the disjunctive optimization in \cite{flegel:2007}. 

\begin{definition}[Stationarity for disjunctive optimization]
    \label{def:disj-c-stat}
  A feasible point $x \in M^{\text{D}(t)}$ of $\text{D}(t)$ is called C-stationary if there exist multipliers 
\[
\zeta_{1,j},\zeta_{2,j}, j \in \mathcal{H}_{12}\left(x\right), \quad \eta_{1,j}, j \in \mathcal{H}_{1}\left(x\right), \quad \eta_{2,j}, j \in \mathcal{H}_{2}\left(x\right), \quad \nu_{1,j}, j \in \mathcal{N}_{1}\left(x\right), \quad \nu_{2,j}, j \in \mathcal{N}_{2}\left(x\right),
\]
such that
\begin{equation}
   \label{eq:disj-c-stat-1} 
   \begin{array}{rcl}
\displaystyle
\nabla f\left( x\right) &=& 
-\displaystyle\sum\limits_{j\in \mathcal{H}_{12}(x)}
\left(\zeta_{1,j} \nabla F_{1,j}\left(x\right) + \zeta_{2,j}\nabla F_{2,j}\left(x\right)\right)\\&&
-\displaystyle\sum\limits_{j\in \mathcal{H}_{1}(x)}
\eta_{1,j} \nabla F_{1,j}\left(x\right)
-\displaystyle\sum\limits_{j\in \mathcal{H}_{2}(x)}
\eta_{2,j} \nabla F_{2,j}\left(x\right)\\&&
+ \displaystyle\sum\limits_{j\in \mathcal{N}_{1}(x)}   \nu_{1,j} \nabla F_{1,j}\left(x\right)
+ \displaystyle\sum\limits_{j\in \mathcal{N}_{2}(x)}  \nu_{2,j} \nabla F_{2,j}\left(x\right).
\end{array}
\end{equation}
\begin{equation}
 \label{eq:disj-c-stat-2}    
 \begin{array}{c}
 \zeta_{1,j}\ge 0,\zeta_{2,j}\ge 0, j \in\mathcal{H}_{12}(x),\\ \\
\eta_{1,j}\ge 0, j \in\mathcal{H}_{1}(x),
\eta_{2,j}\ge 0, j \in\mathcal{H}_{2}(x),
\\ \\ 
\nu_{1,j}\ge 0, j \in\mathcal{N}_{1}(x),
\nu_{2,j}\ge 0, j \in\mathcal{N}_{2}(x).
    \end{array}
\end{equation}
Moreover, the C-stationary point $x$ is called:
\begin{itemize}
\item M-stationary if 
\begin{equation}
       \label{eq:disj-mstat-2}  \zeta_{1,j}\cdot \zeta_{2,j}= 0, j \in\mathcal{H}_{12}(x);
\end{equation}
    \item S-stationary if 
\begin{equation}
       \label{eq:disj-sstat-2}  \zeta_{1,j}=\zeta_{2,j}= 0, j \in\mathcal{H}_{12}(x).
\end{equation}
\end{itemize}
\end{definition}

For a C-stationary point $x \in M^{\text{D}(t)}$ with multipliers $(\zeta,\eta, \nu)$ -- which are unique under DISJ-LICQ -- it is convenient to define the appropriate Lagrange function as
\[
 \begin{array}{rcl}
L^{\text{D}(t)}(x)&=&f(x) -\displaystyle\sum\limits_{j\in \mathcal{H}_{12}(x)}
\left(\zeta_{1,j} \left(t-F_{1,j}\left(x\right)\right) + 
\zeta_{2,j} \left(t-F_{2,j}\left(x\right)\right)\right)\\&&
-\displaystyle\sum\limits_{j\in \mathcal{H}_{1}(x)}
\eta_{1,j} \left(t- F_{1,j}\left(x\right)\right)
-\displaystyle\sum\limits_{j\in \mathcal{H}_{2}(x)}
\eta_{2,j} \left(t- F_{2,j}\left(x\right)\right)\\&&
- \displaystyle\sum\limits_{j\in \mathcal{N}_{1}(x)}   \nu_{1,j} F_{1,j}\left(x\right)
- \displaystyle\sum\limits_{j\in \mathcal{N}_{2}(x)}  \nu_{2,j} F_{2,j}\left(x\right).
\end{array}
\]
The corresponding tangent space is given by
\[
    \mathcal{T}^{\text{D}(t)}_{x}=\left\{
\xi \in \R^{n}\,\left\vert\, \begin{array}{l}
\nabla^T F_{1,j}\left(x\right)\xi=0, j \in \mathcal{H}_{12}\left(x\right)\cup \mathcal{H}_{1}\left(x\right) \cup \mathcal{N}_1\left(x\right),\\
\nabla^T F_{2,j}\left(x\right)\xi=0, j \in \mathcal{H}_{12}\left(x\right)\cup \mathcal{H}_{2}\left(x\right) \cup \mathcal{N}_2\left(x\right)
\end{array}
\right.\right\}.
\]
Let us recall the  definitions of a nondegenerate C-stationary point and its C-index in the context of disjunctive optimization, see \cite{jongen:1997}. Analogously, they turn out to be topologically relevant in the sense of Morse theory.

\begin{definition}[Nondegenerate C-stationarity for disjunctive optimization]
A C-stationary point $x \in M^{\text{D}(t)}$ of $\text{D}(t)$ with multipliers $(\zeta, \eta, \nu)$
is called nondegenerate if 

DISJ-ND1: DISJ-LICQ holds at $x$;

DISJ-ND2: the multipliers for biactive  disjunctive and active inequality constraints do not vanish, i.e.~$\zeta_{1,j}>0$, $\zeta_{2,j}>0$, $j \in \mathcal{H}_{12}(x)$, $\eta_{1,j}>0$, $j \in \mathcal{H}_{1}(x)$, $\eta_{2,j}>0$, $j \in \mathcal{H}_{2}(x)$, $\nu_{1,j} >0$, $j \in \mathcal{N}_1(x)$, and $\nu_{2,j} >0$, $j \in \mathcal{N}_2(x)$;

DISJ-ND3: the restricted Hessian matrix $D^2 L^{\text{D}(t)}(x)\restriction_{\mathcal{T}^{\text{D}(t)}_{x}}$ is nonsingular.
\end{definition}

\begin{definition} [C-index for disjunctive optimization]
    Let $x \in M^{\text{D}(t)}$ be a nondegenerate C-stationary point of $\text{D}(t)$ with unique multipliers $(\zeta,\eta, \nu)$.
The number of negative eigenvalues of the matrix $D^2 L^{\text{D}(t)}(x)\restriction_{\mathcal{T}^{\text{D}(t)}_{x}}$ is called its quadratic index ($\mbox{DISJ-}QI$). The number of nonvanishing biactive disjunctive multipliers $ \zeta_{1,j}, \zeta_{2,j}, j\in \mathcal{H}_{12}(x)$ with $\zeta_{1,j}\not =0, \zeta_{2,j} \not =0$ is called the biactive index ($\mbox{DISJ-}BI$) of $x$. We define the C-index as the sum of both, i.e.~$\mbox{DISJ-}CI=\mbox{DISJ-}QI+\mbox{DISJ-}BI$.
\end{definition}


\section{Main results}
\label{sec:main}


First, let us study the limiting behavior of C- and M-stationary points of the disjunctive regularization $\text{D}(t)$. Under MPCC-LICQ, they turn out to converge to their counterparts of MPCC.

\begin{theorem}[Convergence for C- and M-stationarity]
\label{thm:t-sequence}
    Suppose a sequence of C- or M-stationary points $x^{t} \in M^{\text{D}(t)}$ of $\text{D}(t)$ with multipliers $\left(\zeta ^t, \eta^t, \nu^t\right)$ converges to $\bar x$ for $t \to 0$. Let MPCC-LICQ be fulfilled at $\bar x \in M$. 
    Then, $\bar x$ is a C- or M-stationary point of MPCC, respectively.
\end{theorem}
\begin{proof}
Let us rename the multipliers $(\zeta^t, \eta^t, \nu^t)$ of $x^t$ as follows
\[
\mu^t_{1,j}=\left\{
\begin{array}{ll}
 -\zeta^t_{1,j}    &\mbox{for } j \in \mathcal{H}_{12}\left(x^t\right), \\
 -\eta^t_{1,j}    &\mbox{for } j \in \mathcal{H}_{1}\left(x^t\right), \\
\nu^t_{1,j}    &\mbox{for } j \in \mathcal{N}_{1}\left(x^t\right),\\
0&\mbox{else},
\end{array}
\right.
\qquad
\mu^t_{2,j}=\left\{
\begin{array}{ll}
 -\zeta^t_{2,j}    &\mbox{for } j \in \mathcal{H}_{12}\left(x^t\right), \\
 -\eta^t_{2,j}    &\mbox{for } j \in \mathcal{H}_{2}\left(x^t\right), \\
\nu^t_{2,j}    &\mbox{for } j \in \mathcal{N}_{2}\left(x^t\right),\\
0&\mbox{else}.
\end{array}
\right.
\]
Hence, we get in view of (\ref{eq:disj-c-stat-1})
\[
\nabla f\xt = \displaystyle \sum\limits_{j=1}^{\kappa} \mu_{1,j} \nabla F_{1,j}\xt
+ \sum\limits_{j=1}^{\kappa} \mu_{2,j} \nabla F_{2,j}\xt.
\]
Next, we use that the sets $a_{01}\left(\bar x\right)$, $a_{10}\left(\bar x\right)$, and
$a_{00}\left(\bar x\right)$ are disjoint to obtain
\[
\begin{array}{rcl} 
\nabla f\xt &=& \displaystyle \sum\limits_{j \in a_{01}\left(\bar x\right)} \mu^t_{1,j} \nabla F_{1,j}\xt + \sum\limits_{j \in a_{10}\left(\bar x\right)} \mu^t_{1,j} \nabla F_{1,j}\xt\\
&&\displaystyle +\sum\limits_{j \in a_{01}\left(\bar x\right)} \mu^t_{2,j} \nabla F_{2,j}\xt + \sum\limits_{j \in a_{10}\left(\bar x\right)} \mu^t_{2,j} \nabla F_{2,j}\xt\\
&&\displaystyle +\sum\limits_{j \in a_{00}\left(\bar x\right)} \left(\mu^t_{1,j} \nabla F_{1,j}\xt + \mu^t_{2,j} \nabla F_{2,j}\xt\right).
\end{array}
\]
Let us assume $\mu^t_{1,j}\ne 0$ for $j \in a_{10}\left(\bar x\right)$ for arbitrarily small $t>0$.
Hence, from the definition of $\mu^t_{1,j}$ we get $j \in \mathcal{H}_{12}\left(x^t\right)\cup \mathcal{H}_{1}\left(x^t\right) \cup \mathcal{N}_{1}\left(x^t\right)$. 
However, in the proof of Theorem \ref{thm:LICQ} it was shown that $\mathcal{H}_{12}\xt\cup\mathcal{H}_{1}\xt \cup \mathcal{N}_{1}\xt \subset a_{01}\left(\bar x\right) \cup a_{00}\left(\bar x\right)$, thus, contradicting $j \in a_{10}\left(\bar x\right)$.
Analogously, we derive $\mu^t_{2,j}= 0$ for $j \in a_{01}\left(\bar x\right)$ for all sufficiently small $t >0$.

By renaming the multiplier again as follows
\[
\begin{array}{rcllrcll}
\sigma^t_{1,j}&=&\mu^t_{1,j}    &\mbox{for } j \in a_{01}\left(\bar x\right),&
\sigma^t_{2,j}&=&\mu^t_{2,j}    &\mbox{for } j \in a_{10}\left(\bar x\right),
\\
\varrho^t_{1,j}&=&\mu^t_{1,j}    &\mbox{for } j \in a_{00}\left(\bar x\right),&
\varrho^t_{2,j}&=&\mu^t_{2,j}    &\mbox{for } j \in a_{00}\left(\bar x\right),
\end{array}
\]
we arive at
\[
\begin{array}{rcl} 
\nabla f\xt &=& \displaystyle \sum\limits_{j \in a_{01}\left(\bar x\right)} \sigma^t_{1,j} \nabla F_{1,j}\xt 
+ \sum\limits_{j \in a_{10}\left(\bar x\right)} \sigma^t_{2,j} \nabla F_{2,j}\xt\\
&&\displaystyle +\sum\limits_{j \in a_{00}\left(\bar x\right)} \left(\varrho^t_{1,j} \nabla F_{1,j}\xt + \varrho^t_{2,j} \nabla F_{2,j}\xt\right).
\end{array}
\]
The multipliers $\left(\sigma^t,\varrho^t\right)$ converge to a vector $\left(\bar \sigma, \bar \varrho\right)$ for $t \to 0$ in view of MPCC-LICQ at $\bar x$. Thus, $\bar x$ together with $(\bar \sigma, \bar \varrho)$ fulfills (\ref{eq:cstat-1}).

For the C-stationarity part, we claim that $\bar \varrho_{1,j} \cdot \bar \varrho_{2,j} \ge 0$ for all $j \in a_{00}\left(\bar x\right)$ holds. To prove this, we show $\mu^t_{1,j}\cdot \mu^t_{2,j}\ge 0$ for all $j \in a_{00}\left(\bar x\right)$.
It follows from the definition of $\mu^t$ and (\ref{eq:disj-c-stat-2}) that a possible violation implies the existence of $j \in a_{00}\left(\bar x\right)$ with $j \in \mathcal{N}_{1}\left(x^t\right)\cap \left(\mathcal{H}_{12}\left(x^t\right) \cup \mathcal{H}_{2}\left(x^t\right)\right)$ or
$j \in \mathcal{N}_{2}\left(x^t\right)\cap \left(\mathcal{H}_{12}\left(x^t\right) \cup \mathcal{H}_{1}\left(x^t\right)\right)$.
First, we assume $j \in \mathcal{N}_{1}\left(x^t\right)\cap \left(\mathcal{H}_{12}\left(x^t\right) \cup \mathcal{H}_{2}\left(x^t\right)\right)$.
Thus, $F_{1,j}\xt=0$ and $F_{1,j}\xt \ge t$, a contradiction. Analogously,
$j \in \mathcal{N}_{1}\left(x^t\right)\cap \left(\mathcal{H}_{12}\left(x^t\right) \cup \mathcal{H}_{2}\left(x^t\right)\right)$ leads to a contradiction as well.

Let us now assume that $x^t$ are M-stationary points of $\text{D}(t)$, 
i.e.~(\ref{eq:disj-mstat-2}) additionally holds.
We claim (\ref{eq:mstat-2}) holds. Assume there exists $j \in a_{00}\left(\bar x\right)$ with $\bar \varrho_{1,j} \cdot \bar \varrho_{2,j} \neq 0$. Thus, $\mu^t_{1,j}\cdot \mu^t_{2,j}\ne0$ must hold as well. From the definition of $\mu^t$ and (\ref{eq:disj-mstat-2}) we see $j \in \mathcal{H}_1\xt\cup\mathcal{N}_1\xt$ and $j\in \mathcal{H}_2\xt\cup\mathcal{N}_2\xt$. However, $\mathcal{H}_1\xt\cap \left(\mathcal{H}_2\xt\cup\mathcal{N}_2\xt\right)=\emptyset$. Thus, $j\in \mathcal{N}_1\xt$ and likewise $j\in \mathcal{N}_2\xt$. Using the aforementioned and (\ref{eq:disj-c-stat-2}) we get $\mu^t_{1,j}>0$ and $\mu^t_{2,j}>0$. Taking the limit and recalling 
$\bar \varrho_{1,j} \cdot \bar \varrho_{2,j} \neq 0$, we get  
$\bar \varrho_{1,j} >0$ and $ \bar \varrho_{2,j} > 0$. Thus, $\bar x$ is an M-stationary point of MPCC.
    \qed
\end{proof}

Unfortunately, without additional assumptions the convergence to S-stationary points fails in general. This becomes clear from Example \ref{ex:mintosaddle}.

\begin{example}[Convergence of S-stationary points fails]
\label{ex:mintosaddle}
We consider the following disjunctive regularization with $n=4$ and $\kappa=2$:
\[
\begin{array}{rrl}
\text{D}(t):&\min\limits_{x}& -x_1-x_1x_2+\frac{1}{2}x_3^2+(x_4-1)^2\\
&\mbox{s.\,t.} &\max\left\{t-x_1,t-x_2\right\} \ge 0, x_1\ge 0, x_2\ge 0,\\
&&\max\left\{t-(x_2-x_3),t-x_4\right\} \ge 0, x_2-x_3\ge 0, x_4\ge 0.
\end{array}
\]
The point $x^t=(t,2t,t,1)$ is an S-stationary point of $\text{D}(t)$ for $t<1$. Indeed, it holds
\[
\begin{pmatrix}
    -1-2t\\
    -t\\
    t\\
    0
\end{pmatrix}
=
-\eta^t_{1,1}
\begin{pmatrix}
    1\\
    0\\
    0\\
    0
\end{pmatrix}
-\eta^t_{1,2}
\begin{pmatrix}
    0\\
    1\\
    -1\\
    0
\end{pmatrix}
\]
with the unique multipliers $\eta^t_{1,1}=1+2t$ and $\eta^t_{1,2}=t$.
Note, that $x^t$ is in fact a nondegenerate minimizer of $\text{D}(t)$. Clearly, DISJ-ND1 and DISJ-ND2 are fulfilled. Let us check DISJ-ND3. We have for its Hessian
\[
D^2L^{\text{D}(t)}\xt=\begin{pmatrix}
    0&-1&0&0\\
    -1&0&0&0\\
    0&0&1&0\\
    0&0&0&2
\end{pmatrix}.
\]
For its tangent space we have
\[
\mathcal{T}_{x^t}^{\text{D}(t)}=\left\{\xi \in \R^4\,\left\vert\,\xi_1=0, \xi_2=\xi_3\right.\right\}.
\]
Hence for $\xi \in \mathcal{T}_{x^t}^{\text{D}(t)}$ with $\xi\ne 0$ it holds
\[
\xi^T D^2L^{\text{D}(t)}\xi=-2\xi_1\xi_2+\xi_3^2+2\xi_4^2>0.
\]
Moreover, for the C-index of $x^t$ we also get $\mbox{MPCC-}CI=0$, hence, it is indeed a minimizer.

Let us consider the underlying MPCC:
\[
\begin{array}{rrl}
\mbox{MPCC}:&\min\limits_{x}& -x_1-x_1x_2+\frac{1}{2}x_3^2+(x_4-1)^2\\
&\mbox{s.\,t.} &x_1\cdot x_2= 0, x_1\ge 0, x_2\ge 0,\\
&&(x_2-x_3)\cdot x_4= 0, x_2-x_3\ge 0, x_4\ge 0.
\end{array}
\]
together with $\bar x=(0,0,0,1)$, i.e.~the limiting point of $x^t$ for $t \to 0$.
Since MPCC-LICQ is fulfilled at $\bar x$, it is an M-stationary point of MPCC
according to Theorem \ref{thm:t-sequence}. Indeed, it holds
\[
\begin{pmatrix}
    -1\\
    0\\
    0\\
    0
\end{pmatrix}
=
\bar \sigma_{1,2}
\begin{pmatrix}
    0\\
    1\\
    -1\\
    0
\end{pmatrix}
+
\bar \varrho_{1,1}
\begin{pmatrix}
    1\\
    0\\
    0\\
    0
\end{pmatrix}
+
\bar \varrho_{2,1}
\begin{pmatrix}
    0\\
    1\\
    0\\
    0
\end{pmatrix}
\]
with the unique multipliers
$\bar \sigma_{1,2}=0$, $\bar \varrho_{1,1}=-1$, and $\bar \varrho_{2,1}=0$. However, $\bar x$ is not S-stationary and, thus, also not a minimizer. In fact, it holds
$f\left(x_{\varepsilon}\right)<f\left(\bar x\right)$ for $x_{\varepsilon}=(\varepsilon,0,0,1), \varepsilon>0$.
Moreover, $\bar x$ is degenerate by violating MPCC-ND2 as an C-stationary point of MPCC, more precisely it is a saddle point.
\qed
\end{example}

As an additional generic assumption, which guarantees the convergence to S-stationary points, MPCC-ND2 can be taken. 

\begin{corollary}
    [Convergence for S-stationarity]
\label{cor:s-sequence}
    Suppose a sequence of M-stationary points $x^{t} \in M^{\text{D}(t)}$ of $\text{D}(t)$ with multipliers $\left(\zeta ^t, \eta^t, \nu^t\right)$ converges to $\bar x$ for $t \to 0$. Let MPCC-LICQ and MPCC-ND2 be fulfilled at $\bar x \in M$. 
    Then, $\bar x$ is an S-stationary point of MPCC.
\end{corollary}
\begin{proof}
 In view of Theorem \ref{thm:t-sequence}, it is only left to show that (\ref{eq:sstat-2}) holds at $\bar x$.
 However, we know that (\ref{eq:mstat-2}) is valid. Due to MPCC-ND2, we have
 $\bar \varrho_{1,j}\cdot \bar \varrho_{2,j} \ne 0$. Thus, the assertion follows.

    \qed
\end{proof}

We point out that condition MPCC-ND2 in Corollary \ref{cor:s-sequence} can be weakened. In Proposition \ref{thm:s-statls}, other sufficient conditions in terms of the approaching sequence are formulated. 

\begin{proposition}[Finer convergence for S-stationarity]
\label{thm:s-statls}
        Suppose a sequence of S-stationary points $x^{t} \in M^{\text{D}(t)}$ of $\text{D}(t)$ with multipliers $\left(\zeta ^t, \eta^t, \nu^t\right)$ converges to $\bar x$ for $t \to 0$. Let MPCC-LICQ  be fulfilled at $\bar x \in M$. Additionally, assume that for all sufficiently small $t$ it holds:
        \begin{equation}
    \label{eq:conditionls}
    \begin{array}{lll}
        \lim\limits_{t\to 0}\eta_{1,j}^t=0 & \text{ for all } j \in \mathcal{H}_{1}\xt \cap a^0_{00}\left(\bar x\right),&\\
                 \lim\limits_{t\to 0}\eta_{2,j}^t=0 & \text{ for all } j \in \mathcal{H}_{2}\xt \cap a^0_{00}\left(\bar x\right).&
        \end{array}
\end{equation}
    Then, $\bar x$ is an S-stationary point of MPCC. Moreover, (\ref{eq:conditionls}) is trivially satisfied if 
            \begin{equation}
    \label{eq:conditionls2}
    \left(\mathcal{H}_{1}\xt \cup \mathcal{H}_{2}\xt\right) \cap a^0_{00}\left(\bar x\right)=\emptyset.
\end{equation}
Further, condition (\ref{eq:conditionls2}) itself holds under MPCC-ND2. 
    \end{proposition}
\begin{proof}
By virtue of Theorem \ref{thm:t-sequence}, we have that $\bar x$ is an M-stationary point of MPCC. 
Suppose for some $j \in a_{00}\left(\bar x\right)$ it holds
$\bar \varrho_{1,j}\cdot \bar \varrho_{2,j}=0$. As by definition we have $j \in a^0_{00}\left(\bar x\right)$. Hence, at least one of the multipliers $\bar \varrho_{1,j},\bar \varrho_{2,j}$ vanishes. 
Clearly, $\bar x$ is an S-stationary point for MPCC whenever the potentially nonvanishing multiplier is positive. We consider the remaining cases. First, let $\bar \varrho_{1,j} <0, \bar \varrho_{2,j}=0$. This yields $\mu^t_{1,j}<0$, where $\mu^t$ is defined as in the proof of Theorem \ref{thm:t-sequence}. This implies $j \in \mathcal{H}_{12}\xt \cup \mathcal{H}_1\xt$. However, if it holds $j \in \mathcal{H}_{12}\xt$, then from the proof of Theorem \ref{thm:t-sequence} and due to the S-stationarity of $x^t$:
\[
\bar \varrho_{1,j}=-\lim\limits_{t \to 0}\zeta^t_{1,j}=0,
\] 
which contradicts $\bar \varrho_{1,j} <0$. Therefore, $j \in\mathcal{H}_1\xt$. But then, from 
condition (\ref{eq:conditionls}) we derive
\[
\bar \varrho_{1,j}=-\lim\limits_{t \to 0}\eta^t_{1,j}=0,
\] 
a contradiction to $\bar \varrho_{1,j} <0$ again.
We conclude that $\bar \varrho_{1,j} \ge 0$.
Similarly, we get $\bar \varrho_{2,j} \ge 0$. Altogether, $\bar x$ is an S-stationary point of MPCC.
    \qed
\end{proof}


Next Corollaries \ref{cor:activeindex:disj}
and \ref{cor:multipliers:disj} will be needed as the main technical ingredients for the proofs of our convergence and well-posedness results. Corollary \ref{cor:activeindex:disj} 
relates the active index sets of $\text{D}(t)$ and MPCC. Corollary \ref{cor:multipliers:disj} relates the multipliers of the corresponding C-stationary points. The latter follows straightforwardly from Corollary \ref{cor:activeindex:disj} and the proof of Theorem \ref{thm:t-sequence}.

\begin{corollary}[Active index sets]
\label{cor:activeindex:disj}
Under the assumptions of Theorem \ref{thm:t-sequence} for C-stationarity, we have for the active index subsets of the multipliers $(\bar \sigma, \bar \varrho)$ of $\bar x$ for all sufficiently small $t$:
\begin{itemize}
    \item[a)]  $a_{01}^-\left(\bar x\right)\subset \mathcal{H}_{1}\left(x^t\right)$,
    \item[b)] $a_{01}^+\left(\bar x\right)\subset \mathcal{N}_{1}\left(x^t\right)\backslash \mathcal{N}_{2}\left(x^t\right)$ ,
    \item[c)]  $a_{10}^-\left(\bar x\right)\subset \mathcal{H}_{2}\left(x^t\right)$,
    \item[d)] $a_{10}^+\left(\bar x\right)\subset \mathcal{N}_{2}\left(x^t\right)\backslash \mathcal{N}_{1}\left(x^t\right)$,
    \item[e)]  $a_{00}^-\left(\bar x\right)\subset \mathcal{H}_{12}\left(x^t\right)$,
    \item[f)] $a_{00}^+\left(\bar x\right)\subset \mathcal{N}_{1}\left(x^t\right)\cap \mathcal{N}_{2}\left(x^t\right)$.
\end{itemize}
If additionally $\bar x$ 
fulfills MPCC-ND2, then inclusions e) and f) become equalities.
If additionally $\bar x$ 
fulfills MPCC-ND2 and MPCC-ND4,
then also inclusions a)--d) become equalities.
\end{corollary}
\begin{proof}
a) Suppose $j \in a_{01}^-\left(\bar x\right)$. From the proof of Theorem \ref{thm:t-sequence}
it follows from $\bar \sigma_{1,j}<0$ that $j\in \mathcal{H}_{12}\xt \cup \mathcal{H}_1\xt$.
Let us assume that there exists a subsequence $x^{t_k}$ such that $j \in \mathcal{H}_{12}\left(x^{t_k}\right)$. Then, $F_{2,j}\left(x^{t_k}\right)=t_k$ and, thus, $F_{2,j}\left(x^{t_k}\right)\to 0$ for $t_k\to 0$. This, however, contradicts $j \in a_{01}\left(\bar x\right)$.
Hence, $j\in \mathcal{H}_{1}\xt$ for all $t$ sufficiently small. 

b) As in a), from the proof of Theorem \ref{thm:t-sequence} it follows that $j \in a_{01}^+\left(\bar x\right)$ yields $j \in \mathcal{N}_1\xt$. Assume that it also holds
$j \in \mathcal{N}_2\left(x^{t_k}\right)$ along some subsequence $x^{t_k}$. However, we would have then $\lim\limits_{t_k \to 0}F_{2,j}\left(x^{t_k}\right)=0$ and thus 
$j \notin a_{01}^+\left(\bar x\right)$. The assertion follows directly.

c) and d) are proven analogously to a) and b) due to symmetry.

e) Suppose $j \in a_{00}^-\left(\bar x\right)$. From the proof of Theorem \ref{thm:t-sequence}
it follows from $\bar \varrho_{1,j}<0$ that $j\in \mathcal{H}_{12}\xt \cup \mathcal{H}_1\xt$.
Let us assume that $j \in \mathcal{H}_1\xt$. Then, $F_{2,j}\xt>t$ and, thus, $j \notin
\mathcal{H}_{12}\xt \cup \mathcal{H}_2\xt$. However, this contradicts $\bar \varrho_{2,j}<0$, again by the proof of Theorem \ref{thm:t-sequence}.
Hence, $j\in \mathcal{H}_{12}\xt$. 

f) We again refer to the proof of Theorem \ref{thm:t-sequence} to conclude that from
$\varrho_{1,j}>0$ it follows $j \in \mathcal{N}_1\xt$ and from 
$\varrho_{2,j}>0$ it follows $j \in \mathcal{N}_2\xt$.

Let us assume that MPCC-ND2 and MPCC-ND4 hold at $\bar x$.
Then, inclusions in a)--f) become equalities due to the fact that 
$a^-_{01}\cup a^+_{01} \cup a^-_{10}\cup a^+_{10} \cup a^-_{00} \cup a^+_{00}=\{1,\ldots,\kappa\}$, where all sets on the left side are pairwise disjoint, and the sets
$\mathcal{H}_{12}\left(x^t\right)$, 
$\mathcal{H}_{1}\left(x^t\right)$, 
$\mathcal{H}_{2}\left(x^t\right)$, 
$\mathcal{N}_{1}\left(x^t\right)\backslash \mathcal{N}_{2}\left(x^t\right)$, $\mathcal{N}_{2}\left(x^t\right)\backslash \mathcal{N}_{1}\left(x^t\right)$,
$\mathcal{N}_{1}\left(x^t\right)\cap \mathcal{N}_{2}\left(x^t\right)$ 
being pairwise disjoint.

It is left to show that MPCC-ND4 can be omitted to attain equalities in e) and f). 
Let us start with e). Suppose $j \in \mathcal{H}_{12}\xt \backslash a_{00}^-\left(\bar x\right)$. In particular, 
 $F_{1,j}\xt=F_{2,j}\xt=t$, and, hence, by taking $t \to 0$, $j \in a_{00}\left(\bar x\right)$. Due to MPCC-ND2, we actually have $j \in a_{00}^+\left(\bar x\right)$. This contradicts f). The equality in f) can be shown in a similar way.
\qed
\end{proof}

\begin{corollary}[Multipliers]
\label{cor:multipliers:disj}
If additionally $\bar x$ 
fulfills MPCC-ND2 in Corollary \ref{cor:activeindex:disj}, then for the multipliers it holds:
\[
\begin{array}{lcllcll}
\lim\limits_{t \to 0}-\zeta^t_{1,j}&=&
\left\{
\begin{array}{ll}
\bar \varrho_{1,j}& \mbox{for } j \in a^-_{00}\left(\bar x\right),\\
0 & \mbox{else},
\end{array}\right.&
\lim\limits_{t \to 0}-\zeta^t_{2,j}&=&
\left\{
\begin{array}{ll}
\bar \varrho_{2,j}& \mbox{for } j \in a^-_{00}\left(\bar x\right),\\
0 & \mbox{else},
\end{array}\right.&
\\ \\
\lim\limits_{t \to 0}-\eta^t_{1,j}&=&
\left\{
\begin{array}{ll}
\bar \sigma_{1,j}& \mbox{for } j \in a^-_{01}\left(\bar x\right),\\
0 & \mbox{else},
\end{array}\right.&
\lim\limits_{t \to 0}-\eta^t_{2,j}&=&
\left\{
\begin{array}{ll}
\bar \sigma_{2,j}& \mbox{for } j \in a^-_{10}\left(\bar x\right),\\
0 & \mbox{else},
\end{array}\right.\\ \\
\lim\limits_{t \to 0}\nu^t_{1,j}&=&
\left\{
\begin{array}{ll}
\bar \sigma_{1,j}& \mbox{for } j \in a^+_{01}\left(\bar x\right),\\
\bar \varrho_{1,j}& \mbox{for } j \in a^+_{00}\left(\bar x\right),\\
0 & \mbox{else},
\end{array}\right.&
\lim\limits_{t \to 0}\nu^t_{2,j}&=&
\left\{
\begin{array}{ll}
\bar \sigma_{2,j}& \mbox{for } j \in a^+_{10}\left(\bar x\right),\\
\bar \varrho_{2,j}& \mbox{for } j \in a^+_{00}\left(\bar x\right),\\
0 & \mbox{else}.
\end{array}\right.
\end{array}
\]
\end{corollary}



Now, we are ready to state our first main result on the convergence of C-stationary points of $\text{D}(t)$ and on how the index of the limiting C-stationary point of MPCC may change.

\begin{theorem}[Convergence]
\label{thm:ttoc}
Suppose a sequence of nondegenerate C-stationary points $x^{t} \in M^{\text{D}(t)}$ of 
$\text{D}(t)$ with quadratic and biactive indices $\mbox{DISJ-}QI$ and $\mbox{DISJ-}BI$, respectively, converges to $\bar x$ for $t \to 0$. If $\bar x \in M$ is a nondegenerate C-stationary point of MPCC with multipliers $(\bar \sigma, \bar \varrho)$, then we have for its quadratic and biactive indices $\mbox{MPCC-}QI$ and $\mbox{MPCC-}BI$, respectively: 
\[
\max\left\{\mbox{DISJ-}QI - s, 0\right\} \le \mbox{MPCC-}QI\le \mbox{DISJ-}QI, \quad \mbox{MPCC-}BI=\mbox{DISJ-}BI,
\]
where
\[
s = \left\vert\left\{j\in a_{01}\left(\bar x\right)\,\left\vert\,\bar \sigma_{1,j}=0\right.\right\}\right\vert + \left\vert\left\{j\in a_{10}\left(\bar x\right)\,\left\vert\,\bar \sigma_{2,j}=0\right.\right\} \right\vert.
\]
Thus, for its C-index we have:
\begin{equation}
        \label{eq:main-index}
    \max\left\{\mbox{DISJ-}CI - s, 0\right\} \le \mbox{MPCC-}CI\le \mbox{DISJ-}CI.
\end{equation}
If additionally MPCC-ND4 holds at $\bar x$, then $\mbox{MPCC-}QI=\mbox{DISJ-}QI$, and also $\mbox{MPCC-}CI=\mbox{DISJ-}CI$.
\end{theorem}

\begin{proof}
From Corollary \ref{cor:activeindex:disj} we know that $\mathcal{H}_{12}\xt=a^-_{00}\left(\bar x\right)$ for all $t$ sufficiently small. Therefore, the assertion about the biactive index is directly obtained. For the quadratic index, let us compare the numbers of negative eigenvalues of $D^2 L^{\text{D}(t)}\xt \restriction_{\mathcal{T}^{\text{D}(t)}_{x^t}}$ and $D^2 L(\bar x)\restriction_{\mathcal{T}_{\bar x}}$.
First, we have for the Hessian of the former Lagrange function:
\[
 \begin{array}{rcl}
D^2 L^{\text{D}(t)}\xt&=&D^2f\xt + \displaystyle\sum\limits_{j\in \mathcal{H}_{12}\xt}
\left(\zeta^t_{1,j} D^2F_{1,j}\xt + 
\zeta^t_{2,j} D^2 F_{2,j}\xt\right)\\&&
+\displaystyle\sum\limits_{j\in \mathcal{H}_{1}\xt}
\eta^t_{1,j} D^2 F_{1,j}\xt
+\displaystyle\sum\limits_{j\in \mathcal{H}_{2}\xt}
\eta^t_{2,j} D^2 F_{2,j}\xt\\&&
- \displaystyle\sum\limits_{j\in \mathcal{N}_{1}\xt\backslash \mathcal{N}_{2}\xt}   \nu^t_{1,j} D^2 F_{1,j}\xt
- \displaystyle\sum\limits_{j\in \mathcal{N}_{2}\xt \backslash \mathcal{N}_{1}\xt}  \nu^t_{2,j} D^2 F_{2,j}\xt\\&&
- \displaystyle\sum\limits_{j\in \mathcal{N}_{1}\xt \cup \mathcal{N}_{2}\xt}  \left( \nu^t_{1,j} D^2 F_{1,j}\xt + \nu^t_{2,j} D^2 F_{2,j}\xt\right).
\end{array}
\]
Then, we use the inclusions a)-f) in Corollary \ref{cor:activeindex:disj} to link both Hessians:
\[
 \begin{array}{rcl}
D^2 L^{\text{D}(t)}\xt
&{=}&
D^2f\xt + \displaystyle\sum\limits_{j\in a^-_{00}\left(\bar x\right)}
\left(\zeta^t_{1,j} D^2F_{1,j}\xt + 
\zeta^t_{2,j} D^2 F_{2,j}\xt\right)\\&&
+ \displaystyle\sum\limits_{j\in \mathcal{H}_{12}\xt\backslash a^-_{00}\left(\bar x\right)}
\left(\zeta^t_{1,j} D^2F_{1,j}\xt + 
\zeta^t_{2,j} D^2 F_{2,j}\xt\right)\\&&
+\displaystyle\sum\limits_{j\in a^-_{01}\left(\bar x\right)}
\eta^t_{1,j} D^2 F_{1,j}\xt
+\displaystyle\sum\limits_{j\in a^-_{10}\left(\bar x\right)}
\eta^t_{2,j} D^2 F_{2,j}\xt\\&&
+\displaystyle\sum\limits_{j\in \mathcal{H}_{1}\xt\backslash a^-_{01}\left(\bar x\right)}
\eta^t_{1,j} D^2 F_{1,j}\xt
+\displaystyle\sum\limits_{j\in \mathcal{H}_{2}\xt\backslash a^-_{10}\left(\bar x\right)}
\eta^t_{2,j} D^2 F_{2,j}\xt\\&&
- \displaystyle\sum\limits_{j\in a^+_{01}\left(\bar x\right)}   \nu^t_{1,j} D^2 F_{1,j}\xt
- \displaystyle\sum\limits_{j\in a^+_{10}\left(\bar x\right)}  \nu^t_{2,j} D^2 F_{2,j}\xt\\&&
- \displaystyle\sum\limits_{j\in \left(\mathcal{N}_{1}\xt\backslash \mathcal{N}_{2}\xt\right) \backslash a^+_{01}\left(\bar x\right)}   \nu^t_{1,j} D^2 F_{1,j}\xt
- \displaystyle\sum\limits_{j\in \left(\mathcal{N}_{2}\xt \backslash \mathcal{N}_{1}\xt\right) \backslash a^+_{10}\left(\bar x\right)}  \nu^t_{2,j} D^2 F_{2,j}\xt\\&&
- \displaystyle\sum\limits_{j\in a^+_{00}\left(\bar x\right)}  \left( \nu^t_{1,j} D^2 F_{1,j}\xt + \nu^t_{2,j} D^2 F_{2,j}\xt\right)\\&&
- \displaystyle\sum\limits_{j\in \left(\mathcal{N}_{1}\xt \cup \mathcal{N}_{2}\xt\right)\backslash a^+_{00}\left(\bar x\right)}  \left( \nu^t_{1,j} D^2 F_{1,j}\xt + \nu^t_{2,j} D^2 F_{2,j}\xt\right).
\end{array}
\]
Taking the limit and applying Corollary \ref{cor:multipliers:disj}, we get
\[
\lim\limits_{t \to 0} D^2 L^{\text{D}(t)} \xt = D^2 L(\bar x).
\]
As a consequence and due to MPCC-ND3, we have for all $t$ sufficiently small and $\xi \in \R^n$:
\begin{equation}
    \label{eq:xi-neg}
\xi^T D^2 L^{\text{D}(t)}\xt \xi<0  \mbox{ if and only if } \xi^T D^2 L(\bar x) \xi<0.
\end{equation}

Now, we turn our attention to the comparison of the corresponding tangent spaces. From the proof of Theorem \ref{thm:LICQ} we know that 
\[
\mathcal{H}_{12}\xt\cup\mathcal{H}_{1}\xt \cup \mathcal{N}_{1}\xt \subset a_{01}\left(\bar x\right) \cup a_{00}\left(\bar x\right)
\]
and 
\[\mathcal{H}_{12}\xt\cup\mathcal{H}_{2}\xt \cup \mathcal{N}_{2}\xt \subset a_{10}\left(\bar x\right) \cup a_{00}\left(\bar x\right).
\]
Thus, $\mathcal{T}_{\bar x}\subset \mathcal{T}^{\text{D}(t)}_{x^t}$ holds. Together with (\ref{eq:xi-neg}) we obtain $\text{MPCC-}QI\le \text{DISJ-}QI$.
Further, we have by means of inclusion a)-f) from Corollary \ref{cor:activeindex:disj}
\[
a^-_{01}\left(\bar x\right) \cup a^+_{01}\left(\bar x\right) \cup a^-_{00}\left(\bar x\right)\cup a^+_{00}\left(\bar x\right)
\subset
\mathcal{H}_{12}\xt\cup\mathcal{H}_{1}\xt \cup \mathcal{N}_{1}\xt
\]
and
\[
a^-_{10}\left(\bar x\right) \cup a^+_{10}\left(\bar x\right) \cup a^-_{00}\left(\bar x\right)\cup a^+_{00}\left(\bar x\right)
\subset
\mathcal{H}_{12}\xt\cup\mathcal{H}_{2}\xt \cup \mathcal{N}_{2}\xt.
\]
In view of MPCC-ND2, it also holds $a^0_{00}\left(\bar x\right) = \emptyset$. Altogether,
\[
\mathcal{T}_{\bar x}= \mathcal{T}^{\text{D}(t)}_{x^t}\cap
\left\{
\xi \in \R^{n}\,\left\vert\, \begin{array}{l}
\nabla^T F_{1,j}\left(x\right)\xi=0, j \in a^0_{01}\left(\bar x\right),\\
\nabla^T F_{2,j}\left(x\right)\xi=0, j \in a^0_{10}\left(\bar x\right)
\end{array}
\right.\right\}.
\]
Thus, the number of negative eigenvalues of $D^2 L^{\text{D}(t)}\xt \restriction_{\mathcal{T}^{\text{D}(t)}_{x^t}}$ exceeds the number of negative eigenvalues of $D^2 L(\bar x)\restriction_{\mathcal{T}_{\bar x}}$ by at most $\left| a^0_{01}\left(\bar x\right)\right|+\left|a^0_{10}\left(\bar x\right)\right|$.

\qed
\end{proof}

    We emphasize that in absence of MPCC-ND4 the shift of the C-index in Theorem \ref{thm:ttoc} cannot be avoided in general. This is illustrated by Example \ref{ex:NDC4}, where the lower bound for $\mbox{MPCC-}CI$ in (\ref{eq:main-index}) is attained. An analogous phenomenon of index shift has been discovered in context of the Scholtes regularization, see \cite{laemmel:anomalies}.  

\begin{example}[Necessity of MPCC-ND4]
\label{ex:NDC4}
We consider the following disjunctive regularization $\text{D}(t)$ with
$n=3$ and $\kappa=3$: 

\[
\begin{array}{rll}
 \text{D}(t):& \min\limits_x & -x_1-x_2+2x_1x_2+x_3^2\\
 &\mbox{s.t.}& \max\{t-x_1,t-(x_3-2x_2)\} \ge 0, x_1\ge 0, x_3-2x_2 \ge 0,\\
 &&\max\{t-(x_1+x_2), t-(2-x_3)\} \ge 0, x_1+x_2\ge 0, 2-x_3\ge 0,\\
 &&\max\{t-(x_3-1), t- (x_1-x_2+x_3)\} \ge 0, x_3-1 \ge 0, x_1-x_2+x_3 \ge 0.
\end{array}
\]
Let us show that $x^t=\left(\frac{t}{2}, \frac{t}{2},1\right)$ is a nondegenerate C-stationary point of $\text{D}(t)$ if $0<t<1$. Indeed, it is feasible with $\mathcal{H}_{1}\left(x^t\right)=\{2\}$, $\mathcal{N}_1\left(x^t\right)=\{3\}$, and $\mathcal{H}_{12}\left(x^t\right)= \mathcal{H}_{2}\left(x^t\right)= \mathcal{N}_2\left(x^t\right)=\emptyset$. Further, it holds
\[
\begin{pmatrix}
    -1+t\\-1+t\\2
\end{pmatrix}
=
-\eta^t_{1,2}\begin{pmatrix}
    1\\1\\0
\end{pmatrix}
+\nu^t_{1,3}\begin{pmatrix}
    0\\0\\1
\end{pmatrix}
\]
with the unique multipliers $\eta^t_{1,2}=1-t, \nu^t_{1,3}=2$.
Obviously, DISJ-ND1 and DISJ-ND2 are fulfilled.
For the tangent space at $x^t$ we have
\[
\mathcal{T}^{\text{D}(t)}_{x^t}=\left\{\xi \in \R^3\,\left\vert \,\xi_1=-\xi_2,\xi_3=0\right.\right\}.
\]
For the Hessian of the Lagrange function at $x^t$ we get
\[
D^2L^{\text{D}(t)}\left(x^t\right)=
\begin{pmatrix}
    0&2&0\\
    2&0&0\\
    0&0&2
\end{pmatrix}.
\]
Thus, for $\xi \in \mathcal{T}^{\text{D}(t)}_{x^t}$ with $\xi \ne 0$ it holds
\[
\xi^T D^2L^{\text{D}(t)}\left(x^t\right) \xi  =-4\xi_2^2<0.
\]
We conclude that DISJ-ND3 is also fulfilled and, thus, $x^t$ is nondegenerate. For the indices of $x^t$ we have $\mbox{DISJ-}QI=1$ and $\mbox{DISJ-}BI=0$.

Further, the sequence $x^t$ converges to $\bar x=(0,0,1)$ for $t \to 0$. From Example 2 of \cite{laemmel:anomalies} we know that $\bar x$ is a nondegenerate C-stationary point with vanishing C-index, i.e.~$\mbox{MPCC-}QI=0$ and $\mbox{MPCC-}BI=0$.
Here, the sequence $x^t$ of saddle points of $\text{D}(t)$ converges to the minimizer $\bar x$ of MPCC. This is caused by the violation of MPCC-ND4 at $\bar x$.
\qed
\end{example}

Our second main result focuses on the well-posedness of the disjunctive regularization $\text{D}(t)$. Again, the condition MPCC-ND4 becomes crucial here. 

\begin{theorem}[Well-posedness]
\label{thm:wellposedness}
   Let $\bar x \in M$ be a nondegenerate C-stationary point of MPCC with the quadratic and biactive indices $\mbox{MPCC-}QI$ and $\mbox{MPCC-}BI$, respectively. Assume that MPCC-ND4 holds at $\bar x$. Then, for all sufficiently small $t$ there exists a nondegenerate C-stationary point $x^t \in M^{\text{D}(t)}$ of $\text{D}(t)$ within a neighborhood of $\bar x$, which has the same respective quadratic and biactive indices
\[
   \mbox{DISJ-}QI = \mbox{MPCC-}QI, \quad
   \mbox{DISJ-}BI=\mbox{MPCC-}BI.
\]
Thus, for its C-index we have:
\[
   \mbox{DISJ-}CI = \mbox{MPCC-}CI.
\]
   Moreover, for any fixed $t$ sufficiently small, such $x^t$ is the unique C-stationary point of $\text{D}(t)$ in a sufficiently small neighborhood of $\bar x$.
\end{theorem}
\begin{proof}
First, we prove the existence part.
For that, we consider the auxiliary system of equations $G(t,x,\sigma,\varrho) =0$ given by (\ref{eq:ift-stat})-(\ref{eq:ift-f2}), which mimics stationarity and feasibility. For stationarity we use:
\begin{equation}
\label{eq:ift-stat}
   \begin{array}{rcl}
\displaystyle
0 &=& -\nabla f\left( x\right) +
\displaystyle\sum\limits_{j\in a_{01}^-\left(\bar x\right)}
\sigma_{1,j} \nabla F_{1,j}\left(x\right)
+ \displaystyle\sum\limits_{j\in a_{01}^+\left(\bar x\right)}   \sigma_{1,j} \nabla F_{1,j}\left(x\right)
\\&&
+\displaystyle\sum\limits_{j\in a_{10}^-\left(\bar x\right)}
\sigma_{2,j} \nabla F_{2,j}\left(x\right)
+ \displaystyle\sum\limits_{j\in a_{10}^+(\bar x)}   \sigma_{2,j} \nabla F_{2,j}\left(x\right)\\&&
+\displaystyle\sum\limits_{j\in a_{00}^-\left(\bar x\right)}
\left(\varrho_{1,j} \nabla F_{1,j}\left(x\right) + \varrho_{2,j}\nabla F_{2,j}\left(x\right)\right)\\&&
+\displaystyle\sum\limits_{j\in a_{00}^+\left(\bar x\right)}
\left(\varrho_{1,j} \nabla F_{1,j}\left(x\right) + \varrho_{2,j}\nabla F_{2,j}\left(x\right)\right).
\end{array}
\end{equation}
For feasibility we use:
\begin{equation}
\label{eq:ift-f1}
F_{1,j}(x)-t=0,\,j\in a_{01}^-\left(\bar x\right)\cup a_{00}^-\left(\bar x\right), \quad
F_{1,j}\left(x\right)=0,\,j\in a_{01}^+\left(\bar x\right) \cup a_{00}^+\left(\bar x\right),     
\end{equation}
\begin{equation}
\label{eq:ift-f2}
F_{2,j}(x)-t=0,\,j\in a_{10}^-\left(\bar x\right) \cup a_{00}^-\left(\bar x\right), \quad
F_{2,j}\left(x\right)=0,\,j\in a_{10}^+\left(\bar x\right)\cup a_{00}^+\left(\bar x\right).
\end{equation}
Due to feasibility and C-stationarity for MPCC of $\bar x$ with multipliers $(\bar \sigma, \bar \rho)$, as well as MPCC-ND2 and MPCC-ND4, the vector
$(0,\bar x,\bar \sigma,\bar \varrho)$ solves the system of equations (\ref{eq:ift-stat})-(\ref{eq:ift-f2}).

We consider the blockwise matrix
$$
\displaystyle \frac{\partial G (t,x, \sigma, \varrho)}{\partial (x,\sigma,\varrho)}=\begin{bmatrix}
A&B\\
B^T&D
\end{bmatrix}.$$
Here, we have
\[
\begin{array}{rcl}
      A&=& 
\displaystyle -D^2 f\left( x\right)+\sum\limits_{j\in a_{01}^-\left(\bar x\right)}
\sigma_{1,j} D^2 F_{1,j}\left(x\right)
+ \displaystyle\sum\limits_{j\in a_{01}^+\left(\bar x\right)}   \sigma_{1,j} D^2 F_{1,j}\left(x\right)
\\&&
+\displaystyle\sum\limits_{j\in a_{10}^-\left(\bar x\right)}
\sigma_{2,j} D^2 F_{2,j}\left(x\right)
+ \displaystyle\sum\limits_{j\in a_{10}^+\left(\bar x\right)}   \sigma_{2,j} D^2 F_{2,j}\left(x\right)\\&&
+\displaystyle\sum\limits_{j\in a_{00}^-\left(\bar x\right)}
\left(\varrho_{1,j} D^2 F_{1,j}\left(x\right) + \varrho_{2,j} D^2 F_{2,j}\left(x\right)\right)\\&&
+\displaystyle\sum\limits_{j\in a_{00}^+\left(\bar x\right)}
\left(\varrho_{1,j} D^2 F_{1,j}\left(x\right) + \varrho_{2,j}D^2 F_{2,j}\left(x\right)\right).
\end{array}
\]
Due to MPCC-ND2, it holds $a_{00}^0(\bar x) = \emptyset$, and due to MPCC-ND4, $a_{01}^0(\bar x) = a_{10}^0(\bar x) = \emptyset$. Therefore, the columns of $B$ can be rearranged as follows:
\[
\nabla F_{1,j}\left(x\right), j \in a_{01}\left(\bar x\right)\cup a_{00}\left(\bar x\right),
\quad
\nabla F_{2}\left(x\right), j \in a_{10}\left(\bar x\right)\cup a_{00}\left(\bar x\right).
\]
Further, we have $D=0$. Therefore, we can apply Theorem 2.3.2 from \cite{jongen:2004}, which says that
$
\begin{bmatrix}
A&B\\
B^T&0
\end{bmatrix}$ is nonsingular if and only if
$\xi^TA\xi\ne0$ for all $\xi \in B^\perp\backslash \{0\}$. Here, $B^{\perp}$ refers to the orthogonal complement of the subspace spanned by the columns of $B$. 
We note that here $B^\perp=\mathcal{T}_{\bar x}$ at $\left(0, \bar x, \bar \sigma, \bar \varrho\right)$. 
For $\xi \in B^{\perp}$ with $ \xi \ne 0$ we check:
\[
\xi^TA\xi = - \xi^T D^2 L\left(\bar x\right) \xi \overset{\text{MPCC-ND3}}{\ne} 0.
\]
By using the implicit function theorem we obtain for any $t>0$ sufficiently small a solution $\left(t,x^t,\sigma^t,\varrho^t\right)$ of the system of equations (\ref{eq:ift-stat})-(\ref{eq:ift-f2}).

In virtue of (\ref{eq:ift-f1}), (\ref{eq:ift-f2}) and recalling $a_{01}^0(\bar x) = a_{10}^0(\bar x) = a_{00}^0(\bar x) = \emptyset$, $x^t$ are feasible for $\text{D}(t)$, i.e.~$x^t \in M^{\text{D}(t)}$, taking $t$ even smaller if necessary. In particular, it holds:
\begin{itemize}
    \item[(a)] $\mathcal{H}_{12}\left(x^t\right)=a_{00}^-\left(\bar x\right)$,
    \item[(b)] $\mathcal{H}_{1}\left(x^t\right)=a_{01}^-\left(\bar x\right)$,
    \item[(c)] $\mathcal{H}_{2}\left(x^t\right)=a_{10}^-\left(\bar x\right)$,
    \item[(d)] $\mathcal{N}_1\left(x^t\right)=a_{01}^+\left(\bar x\right)\cup a_{00}^+\left(\bar x\right)$,
    \item[(e)] $\mathcal{N}_2\left(x^t\right)=a_{10}^+\left(\bar x\right)\cup a_{00}^+\left(\bar x\right)$. 
    \end{itemize}
We can also ensure by continuity arguments  that
\begin{itemize}
    \item[(i)] $\sigma^t_{1,j}<0$, \,$j\in a^-_{01}\left(\bar x\right)$,
    $\sigma^t_{2,j}<0$, \,$j\in a^-_{10}\left(x^t\right)$,
    \item[(ii)] $\sigma^t_{1,j}>0$, \,$j\in a^+_{01}\left(\bar x\right)$,
    $\sigma^t_{2,j}>0$, \,$j\in a^+_{10}\left(x^t\right)$,
       \item[(iii)] $\varrho^t_{1,j}<0, \varrho^t_{2,j}<0$,\,$j\in a^-_{00}\left(\bar x\right)$,
   \item[(iv)] $\varrho^t_{1,j}>0,\varrho^t_{2,j}>0$,\,$j\in a^+_{00}\left(\bar x\right)$.
\end{itemize} 
Further, we rename the multipliers as follows
\[
\zeta^{t}_{1,j}=\left\{
\begin{array}{ll}
\displaystyle -\varrho^t_{1,j}& \mbox{for }j\in a_{00}^-\left(\bar x\right),\\
0&\mbox{else,}
\end{array}\right.
\quad
\zeta^{t}_{2,j}=\left\{
\begin{array}{ll}
\displaystyle -\varrho^t_{2,j}& \mbox{for }j\in a_{00}^-\left(\bar x\right),\\
0&\mbox{else,}
\end{array}\right.
\]
\[
\eta^{t}_{1,j}=\left\{
\begin{array}{ll}
\displaystyle -\sigma^t_{1,j}& \mbox{for }j\in a_{01}^-\left(\bar x\right),\\
0&\mbox{else,}
\end{array}\right.
\quad
\eta^{t}_{2,j}=\left\{
\begin{array}{ll}
\displaystyle -\sigma^t_{2,j}& \mbox{for }j\in a_{10}^-\left(\bar x\right),\\
0&\mbox{else,}
\end{array}\right.
\]
\[
\nu^t_{1,j}=\left\{
\begin{array}{ll}
\displaystyle \sigma_{1,j}^t& \mbox{for }j\in a_{01}^+\left(\bar x\right),\\ 
\displaystyle \varrho_{1,j}^t& \mbox{for }j\in a_{00}^+\left(\bar x\right),\\ 
0&\mbox{else,}
\end{array}\right.\\
\quad
\nu^t_{2,j}=\left\{
\begin{array}{ll}
\displaystyle \sigma_{2,j}^t& \mbox{for }j\in a_{10}^+\left(\bar x\right),\\ 
\displaystyle \varrho_{2,j}^t& \mbox{for }j\in a_{00}^+\left(\bar x\right),\\ 
0&\mbox{else.}
\end{array}\right.\\
\]
Substituting the latter in (\ref{eq:ift-stat}) and using (a)-(e) together with (i)--(iv), we conclude that $x^t$ is C-stationary for $\text{D}(t)$, i.e.~it fulfills (\ref{eq:disj-c-stat-1}) and (\ref{eq:disj-c-stat-2}) with multipliers $\left(\zeta^t,\eta^t,\nu^t\right)$.  

Due to Theorem \ref{thm:LICQ}, DISJ-ND1 is satisfied for all $t$ sufficiently small. Furthermore, DISJ-ND2 is valid due to the definition of $\left(\zeta^t,\eta^t,\nu^t\right)$ by recalling (a)-(e) and (i)-(iv).
We show that DISJ-ND3 is also fulfilled at $x^t$, i.e.~the restriction of 
$D^2 L^{\text{D}(t)}\left( x^t\right)$ on $\mathcal{T}^{\text{D}(t)}_{x^t}$ is nonsingular. With the same arguments as in the proof of Theorem \ref{thm:ttoc} we derive
\[
\lim\limits_{t \to 0} D^2 L^{\text{D}(t)} \xt= D^2 L(\bar x).
\]
In view of MPCC-ND3, we get for $\xi \in \mathcal{T}_{\bar x}\left(\bar x\right)$ with $\xi \ne 0$:
\[
\xi^T D^2 L^{\text{D}(t)} \xt \xi \ne 0.
\]
We refer again to the proof of Theorem \ref{thm:ttoc}, where by using MPCC-ND2 we derived
\[
\mathcal{T}_{\bar x}= \mathcal{T}^{\text{D}(t)}_{x^t}\cap
\left\{
\xi \in \R^{n}\,\left\vert\, \begin{array}{l}
\nabla^T F_{1,j}\left(x\right)\xi=0, j \in a^0_{01}\left(\bar x\right),\\
\nabla^T F_{2,j}\left(x\right)\xi=0, j \in a^0_{10}\left(\bar x\right)
\end{array}
\right.\right\}.
\]
By virtue of MPCC-ND4, the index sets here are empty, and we have $\mathcal{T}_{\bar x}=\mathcal{T}^{\text{D}(t)}_{x^t}$. Consequently, DISJ-ND3 holds at $x^t$.

We proceed with the uniqueness part. 
For that, let us consider an arbitrary C-stationary point $\widetilde x^t \in M^{\text{D}(t)}$ of $\text{D}(t)$ with multipliers $(\widetilde \zeta^t, \widetilde \eta^t,\widetilde \nu^t)$ in a sufficiently small neighborhood of $\bar x$. 
In view of MPCC-ND2 and MPCC-ND4, we can apply Corollary \ref{cor:activeindex:disj} to obtain a)--f) for the corresponding index sets at $\widetilde x^t$.
Next, we define the multipliers as follows
\[
\begin{array}{lclllcll}
\widetilde \sigma_{1,j}^t&=&-\widetilde \eta_j^t,&j \in a_{01}^-(\bar x), &
\widetilde \sigma_{1,j}^t&=&\widetilde\nu_{1,j}^t,&j \in a_{01}^+(\bar x),\\
\widetilde \sigma_{2,j}^t&=&-\widetilde \eta_j^t,&j \in a_{10}^-(\bar x),&
\widetilde \sigma_{2,j}^t&=&\widetilde\nu_{2,j}^t,&j \in a_{10}^+(\bar x),\\
\widetilde \varrho_{1,j}^t&=&-\widetilde \zeta_j^t  ,&j \in a_{00}^-(\bar x),&
\widetilde \varrho_{1,j}^t&=&\widetilde\nu_{1,j}^t,&j \in a_{00}^+(\bar x),\\
\widetilde \varrho_{2,j}^t&=&-\widetilde \zeta_j^t ,&j \in a_{00}^-(\bar x),&
\widetilde \varrho_{2,j}^t&=&\widetilde\nu_{2,j}^t,&j \in a_{00}^+(\bar x).\\
\end{array}
\]
A straightforward calculation yields that $\left(t, \widetilde x^t, \widetilde \sigma^t, \widetilde \varrho^t\right)$ fulfills equations (\ref{eq:ift-stat})-(\ref{eq:ift-f2}) for $t$ sufficiently small. Since the solution $x^t$ of this system of equations in the neighborhood of $\bar x$ was derived by means of the implicit function theorem, it must be unique, and, hence, $\widetilde x^t=x^t$.

    \qed
\end{proof}

Theorems \ref{thm:ttoc} and \ref{thm:wellposedness} can be easily extended to S-stationary points and, in particular, to local minimizers:
\begin{itemize}
    \item 
If $x^{t} \in M^{\text{D}(t)}$ is a sequence of nondegenerate M-stationary points of 
$\text{D}(t)$ and the S-stationary limiting point $\bar x \in M$ of MPCC is nondegenerate, cf.~Corollary \ref{cor:s-sequence}, then it holds:
\[
\mbox{MPCC-}BI=\mbox{DISJ-}BI=0.
\]
This observation is important for understanding why the disjunctive-type regularization of MPCC is advantageous. For the approximating points $x^t$, the condition $\mbox{DISJ-}BI=0$ means that the corresponding biactive disjunctive constraints are absent, i.e.~$\mathcal{H}_{12}(x^t)=\emptyset$. Hence, at least locally they are usual Karush-Kuhn-Tucker points and do not show nonsmoothness. On the other hand, $\mbox{MPCC-}BI=0$ just means that the biactive multipliers of $\bar x$ are positive, i.e.~$\bar\varrho_{1,j}>0$ and $\bar\varrho_{2,j} > 0$ for all $j\in a_{00}\left(\bar x\right)$. Here, it may well happen that $a_{00}\left(\bar x\right)\not=\emptyset$ and, thus, nonsmoothness prevails, see Example \ref{ex:biactive}. We note that, although the biactive indices remain vanishing, the index shift of the quadratic part can also occur for S-stationary points if MPCC-ND4 is violated, see Example \ref{ex:NDC4}. In general, for S-stationarity we again have
\[
\max\left\{\mbox{DISJ-}QI - s, 0\right\} \le \mbox{MPCC-}QI\le \mbox{DISJ-}QI.
\]
Applied to nondegenerate local minimizers $x^t \in M^{\text{D}(t)}$ of $\text{D}(t)$, for which $\mbox{DISJ-}QI=0$ is known to hold, we obtain from above that $\mbox{MPCC-}QI=0$ for $\bar x$. The latter is true even in the absence of MPCC-ND4. The limiting point $\bar x$ turns out to be then a local minimizer of MPCC, since its C-index vanishes, i.e.~$\mbox{MPCC-}CI=0$.
\item Vice versa, suppose $\bar x \in M$ is a nondegenerate M-stationary point of 
MPCC fulfilling MPCC-ND4.
Then, for all sufficiently small $t$ there exists a nondegenerate S-stationary point $x^t \in M^{\text{D}(t)}$ of $\text{D}(t)$ within a neighborhood of $\bar x$, which has the same respective quadratic and biactive indices
\[
   \mbox{DISJ-}QI = \mbox{MPCC-}QI, \quad
   \mbox{DISJ-}BI=\mbox{MPCC-}BI=0.
\]
Moreover, for any fixed $t$ sufficiently small, such $x^t$ is the unique C-stationary point of $\text{D}(t)$ in a sufficiently small neighborhood of $\bar x$. This result guarantees as above that, due to $\mbox{DISJ-}BI=0$, the biactive disjunctive constraints are absent for $x^t$. It has thus to be actually a Karush-Kuhn-Tucker point of $\text{D}(t)$, at least as viewed locally. Let us additionally assume that $\bar x\in M$ is a nondegenerate local minimizer of MPCC. In particular, $\mbox{MPCC-}QI=0$ holds for $\bar x$, and, we deduce $\mbox{DISJ-}QI=0$ for $x^t$. Since their C-index then vanishes, i.e.~$\mbox{DISJ-}CI=0$, the approximating points $x^t$ turn out to be local minimizers of $\text{D}(t)$.  
\end{itemize}
Overall, we conclude that nondegenerate S-stationary points of MPCC with nonsmooth biactive structure will be approximated by the nondegenerate actually Karush-Kuhn-Tucker points of $\text{D}(t)$. The same applies to the nondegenerate local minimizers of MPCC. While the latter could show nonsmoothness by having $a_{00}(\bar x) \not =\emptyset$, they will be always approximated by smooth local minimizers of $\text{D}(t)$ with $\mathcal{H}_{12}(x^t)=\emptyset$, cf. next Example \ref{ex:biactive}.

\begin{example}
    \label{ex:biactive}
We consider the following disjunctive regularization $\text{D}(t)$ with
$n=2$ and $\kappa=1$: 

\[
\begin{array}{rll}
 \text{D}(t):& \min\limits_x &x_1+x_2\\
 &\mbox{s.t.}& \max\{t-x_1,t-x_2\} \ge 0, x_1\ge 0, x_2 \ge 0.
\end{array}
\]
It is straightforward to check that $x^t=\left(0,0\right)$ is the unique C-stationary point of $\text{D}(t)$. In particular, it is a nondegenerate minimizer. 
Moreover, it holds $\mathcal{H}_{12}\left(x^t\right)=\emptyset$.

Let us consider the underlying MPCC:
\[
\begin{array}{rll}
 \mbox{MPCC}:& \min\limits_x &x_1+x_2\\
 &\mbox{s.t.}& x_1\cdot x_2= 0, x_1\ge 0, x_2 \ge 0
\end{array}
\]
together with the limiting point $\bar x=(0,0)$ of $x^t$ for $t\to 0$.
Clearly, $\bar x$ is the unique C-stationary of MPCC. Further, it is a nondegenerate minimizer fulfilling MPCC-ND4. However, the biactive set 
$a_{00}\left(\bar x\right)=\{1\}$ is not empty.
\qed
\end{example}


\section{Comparison to the existing literature}\label{sec:kanzow}

Let us turn our attention to the Kanzow-Schwartz regularization $\mbox{KS}(t)$ of MPCC from \cite{kanzow:Relax}. It has been suggested there to treat $\text{KS}(t)$ as an NLP instance. This is to say that the standard LICQ and usual Karush-Kuhn-Tucker points are considered for $\text{KS}(t)$. 


\begin{theorem}
    [MPCC-LICQ vs. LICQ, \cite{kanzow:Relax}]
\label{thm:LICQ-KS}
Let a feasible point $\bar x \in M$ of MPCC fulfill MPCC-LICQ. Then, LICQ holds at all feasible points $x \in M^{\text{KS}(t)}$ of $\text{KS}(t)$ with $\mathcal{H}_{12}\left( x\right)=\emptyset$ for all sufficiently small $t$, whenever they are sufficiently close to $\bar x$. 
\end{theorem}
\noindent
We note that LICQ does not hold at any feasible point $ x \in M^{\text{KS}(t)}$ of $\text{KS}(t)$ with $\mathcal{H}_{12}\left( x\right)\ne\emptyset$. This is due to the fact that for all $j \in \mathcal{H}_{12}\left( x\right)$ not only $\Phi_j\left( x,t\right)=0$, but also $\nabla \Phi_j\left( x,t\right)=0$ holds.
The degeneracy of such $ x$ hampers the study of $\text{KS}(t)$ from the perspective of Morse theory.  In particular, to trace the quadratic index of Karush-Kuhn-Tucker points of the Kanzow-Schwartz regularization is in general not possible.
Similarly, the standard LICQ also does not hold at any feasible point $\bar x \in M$ of MPCC with $a_{00}\left(\bar x\right)\ne \emptyset$. From this point of view $\text{KS}(t)$, like MPCC itself, turns out to be a degenerate NLP by construction. 

Furthermore, a direct attempt to write 
$\text{KS}(t)$ in terms of disjunctive optimization does not overcome degeneracy. To see this, we rewrite the feasible set of $\text{KS}(t)$ as follows:
\[
\begin{array}{rcl} 
    M^{\text{KS}(t)}
    &=&\left\{
    x \in\R^n\, \left\vert\,   
    \begin{array}{l} 
   \left(F_{1,j}(x)-t\right)\cdot\left(F_{2,j}(x)-t\right)\le 0 \text{ if }
   F_{1,j}(x)+F_{2,j}(x)-2t\ge0,\\
    F_{1,j}(x) \ge 0, F_{2,j}(x)\ge 0, j=1,\ldots,\kappa 
            \end{array}
    \right. \right\}.
    \end{array}
\]
We notice that for $F_{1,j}(x)+F_{2,j}(x)-2t=0$ it can neither hold
$F_{1,j}(x)>t$ and $F_{2,j}(x)>t$ nor $F_{1,j}(x)<t$ and $F_{2,j}(x)<t$.
Thus, the feasible set simplifies to
\[
\begin{array}{rcl}
    M^{\text{KS}(t)}
    &=&\left\{
    x \in\R^n\, \left\vert\,   
    \begin{array}{l} 
   \left(F_{1,j}(x)-t\right)\cdot\left(F_{2,j}(x)-t\right)\le 0 \text{ if }
   F_{1,j}(x)+F_{2,j}(x)-2t>0,\\
    F_{1,j}(x) \ge 0, F_{2,j}(x)\ge 0, j=1,\ldots,\kappa, 
            \end{array}
    \right. \right\}.
\end{array} 
\]
Finally, we obtain by means of the disjunctive constraints
\begin{equation}
\label{eq:ksdisj}    
\begin{array}{c}
    M^{\text{KS}(t)}
        =\left\{
    x \in\R^n\, \left\vert\,   
    \begin{array}{l} 
   \max\left\{-\left(F_{1,j}(x)-t\right)\cdot\left(F_{2,j}(x)-t\right),
   2t-F_{1,j}(x)-F_{2,j}(x)\right\}\ge0,\\
    F_{1,j}(x) \ge 0, F_{2,j}(x)\ge 0, j=1,\ldots,\kappa 
            \end{array}
    \right. \right\}.\\
    \end{array}
\end{equation}
\noindent
Applying the definition of DISJ-LICQ in \cite{jongen:1997} to the latter feasible set, $x \in M^{\text{KS}(t)}$ fulfills DISJ-LICQ if and only if the following vectors are linearly independent:
\[
-\nabla F_{1,j}\left( x\right)\cdot \left(F_{2,j}\left( x\right)-t\right)-\nabla F_{2,j}\left( x\right)\cdot \left(F_{1,j}\left( x\right)-t\right), j \in \mathcal{H}_{12}\left( x\right)\cup \mathcal{H}_1\left( x\right)\cup \mathcal{H}_2\left( x\right),\]
\[
-\nabla F_{1,j}\left( x\right) -\nabla F_{2,j}\left( x\right), j\in \mathcal{H}_{12}\left( x\right), \quad
\nabla F_{1,j}\left( x\right), j \in \mathcal{N}_1\left( x\right), \quad
\nabla F_{2,j}\left( x\right), j \in \mathcal{N}_2\left( x\right).
\]
However, $-\nabla F_{1,j}\left( x\right)\cdot \left(F_{2,j}\left( x\right)-t\right)-\nabla F_{2,j}\left( x\right)\cdot \left(F_{1,j}\left( x\right)-t\right)$ vanishes for all $j \in \mathcal{H}_{12}\left( x\right)$.
Therefore, DISJ-LICQ is violated at $ x$ whenever $\mathcal{H}_{12}\left( x\right)\ne \emptyset$. We emphasize that the observed violations of constraint qualifications for $\text{KS}(t)$ motivated us to use a new functional description for the regularized feasible set in $\text{D}(t)$.
Note that applying C-stationarity to $\text{KS}(t)$ if viewed as a disjunctive optimization problem via (\ref{eq:ksdisj}) is not promising either. By using the theory of disjunctive optimization from \cite{jongen:1997} once again, it is straightforward to deduce that
a feasible point $x\in M^{\text{KS}(t)}$ is C-stationary for $\text{KS}(t)$
if and only if it is C-stationary for $\text{D}(t)$ and it additionally holds for the corresponding multipliers:
\[
\zeta_{1,j}=\zeta_{2,j}, j\in \mathcal{H}_{12}\left(x\right). 
\]
Therefore, in comparison with $\text{D}(t)$, we only obtain a subset of C-stationary points. Thus, the study of the Kanzow-Schwartz regularization 
$\text{KS}(t)$ from the perspective of Morse theory is also hampered, even viewed as a disjunctive optimization problem. 
Further relations between $\text{D}(t)$ and $\text{KS}(t)$ are highlighted in Proposition \ref{lem:kktands} next.

\color{black}

\begin{proposition}[KKT-points and S-stationarity]
\label{lem:kktands}
A point $ x$ is a Karush-Kuhn-Tucker point of $\text{KS}(t)$ if and only if it is an S-stationary point of $\text{D}(t)$.    
\end{proposition}

\begin{proof}
      Since $ x \in M^{\text{KS}(t)}$ of $\text{KS}(t)$ is a Karush-Kuhn-Tucker point, there exist multipliers 
\[
\mu_{j}, \mu_{1,j}, \mu_{2,j},\,j\in\left\{1,\ldots,\kappa\right\},
\]
such that
\[
   \begin{array}{rcl}
\displaystyle
\nabla f\left(  x\right) &=& -\displaystyle\sum\limits_{j=1}^{\kappa}
\mu_{j} \nabla \Phi_{j}\left( x,t\right) +
\displaystyle\sum\limits_{j=1}^{\kappa}
\mu_{1,j} \nabla F_{1,j}\left( x\right)
+
\displaystyle\sum\limits_{j=1}^{\kappa}
\mu_{2,j} \nabla F_{2,j}\left( x\right),
\end{array}
\]
\[
 \begin{array}{c}
\mu_{j}\ge 0,
\mu_{j}\cdot \Phi_{j}\left( x,t\right)=0,\mu_{1,j}\ge 0, \mu_{1,j}\cdot F_{1,j}\left( x\right)=0, 
\mu_{2,j}\ge 0, \mu_{2,j}\cdot F_{2,j}\left( x\right)=0,
  j=1,\ldots,\kappa.
    \end{array}
\]
As in \cite{kanzow:Relax}, we rename the multipliers
\[
 \eta_{1,j}=  \mu_{j}\left(F_{2,j}\left( x\right)-t\right),
 \eta_{2,j}= \mu_{j}\left(F_{1,j}\left( x\right)-t\right),
 \nu_{1,j}= \mu_{1,j},
 \nu_{2,j}=  \mu_{2,j}, j=1,\ldots,\kappa,
\]
 and use the special structure of $\nabla \Phi$ 
 to obtain the S-stationarity of $ x$ for  $\text{D}(t)$:
\begin{equation}
   \label{eq:ks-kkt-1} 
   \begin{array}{rcl}
\displaystyle
\nabla f\left( x\right) &=&
-\displaystyle\sum\limits_{j\in \mathcal{H}_{1}\left( x\right)}
\eta_{1,j} \nabla F_{1,j}\left( x\right)
-\displaystyle\sum\limits_{j\in \mathcal{H}_{2}\left( x\right)}
\eta_{2,j} \nabla F_{2,j}\left( x\right)\\&&
+ \displaystyle\sum\limits_{j\in \mathcal{N}_{1}\left( x\right)}   \nu_{1,j} \nabla F_{1,j}\left( x\right)
+ \displaystyle\sum\limits_{j\in \mathcal{N}_{2}\left( x\right)}  \nu_{2,j} \nabla F_{2,j}\left( x\right).
\end{array}
\end{equation}
\begin{equation}
 \label{eq:ks-kkt-2}    
 \begin{array}{c}
\eta_{1,j}\ge 0, j \in\mathcal{H}_{1}\left( x\right),
\eta_{2,j}\ge 0, j \in\mathcal{H}_{2}\left( x\right),
\\ \\ 
\nu_{1,j}\ge 0, j \in\mathcal{N}_{1}\left( x\right),
\nu_{2,j}\ge 0, j \in\mathcal{N}_{2}\left( x\right).
    \end{array}
\end{equation}
The reverse implication follows analogously.
\qed

\end{proof}

Concerning the convergence of Karush-Kuhn-Tucker points of $\text{KS}(t)$ to S-stationary points of MPCC, we mention the following result from \cite{kanzow:Relax}.

\begin{theorem}[Convergence for S-stationarity, \cite{kanzow:Relax}]
\label{thm:s-statks}
        Suppose a sequence of Karush-Kuhn-Tucker points $ x^{t} \in M^{\text{KS}(t)}$ of $\text{KS}(t)$ with multipliers $\left(\mu^t, \mu_1^t, \mu_2^t\right)$ converges to $\bar x$ for $t \to 0$. Let MPCC-LICQ  be fulfilled at $\bar x \in M$. Additionally, assume that there exists a subsequence such that it holds:
\begin{equation}
    \label{eq:conditionks}
    F_{1,j}\left( x^t\right)\le t, F_{2,j}\left( x^t\right) \le t \mbox{ for all } j \in a_{00}\left(\bar x\right).
\end{equation}
    Then, $\bar x$ is an S-stationary point of MPCC.
    \end{theorem}

Suppose MPCC-ND2 is fulfilled at $\bar x$. Thus, it holds $a_{00}^0\left(\bar x\right)=\emptyset$.
If additionally MPCC-LICQ holds at $\bar x$, we can use Corollary \ref{cor:activeindex:disj} to conclude that e) and f) hold with equality. As consequence, MPCC-ND2 implies condition (\ref{eq:conditionks}). In turn, condition (\ref{eq:conditionks}) yields 
$\mathcal{H}_1\xt\cap a_{00}\left(\bar x\right)=\mathcal{H}_2\xt \cap a_{00}\left(\bar x\right)=\emptyset$, in particular, condition (\ref{eq:conditionls2}) is valid. Overall, the following implications hold:
\[
\mbox{MPCC-ND2} \quad
\overset{\text{MPCC-LICQ}}{\Longrightarrow}\quad
(\ref{eq:conditionks})\quad {\Longrightarrow} \quad 
(\ref{eq:conditionls2})\quad
\Longrightarrow \quad
(\ref{eq:conditionls}).
\]
Although MPCC-ND2 is the strongest condition among these, its advantage cannot be neglected. In fact, it has to be checked only at the limiting point $\bar x$, whereas the other conditions apply to the entire sequence $x^t$. The latter hampers the applicability of (\ref{eq:conditionls}), (\ref{eq:conditionls2}), or (\ref{eq:conditionks}).
 Nevertheless, we illustrate by means of Example \ref{ex:ls} that both conditions (\ref{eq:conditionls}) and (\ref{eq:conditionls2}) are strictly weaker than (\ref{eq:conditionks}), cf. Proposition \ref{thm:s-statls}.

\begin{example}[Convergence under (\ref{eq:conditionls})]
\label{ex:ls}
We consider the following disjunctive regularization with $n=4$ and $\kappa=2$:
\[
\begin{array}{rrl}
\text{D}(t):&\min\limits_{x}& x_1^2-x_2^2-x_2+x_3^2+x_3+(x_4-1)^2\\
&\mbox{s.\,t.} &\max\left\{t-(x_1-x_2),t-(x_1+x_2)\right\} \ge 0, x_1-x_2\ge 0, x_1+x_2\ge 0,\\
&&\max\left\{t-(x_2-x_3),t-x_4)\right\} \ge 0, x_2-x_3\ge 0, x_4\ge 0.
\end{array}
\]
The point $x^t=(t,t,0,1)$ is an S-stationary point of $\text{D}(t)$ for $t<1$ with
$\mathcal{H}_1\xt=\left\{2\right\}$ and $\mathcal{H}_2\xt=\emptyset$.
Indeed, it holds
\[
\begin{pmatrix}
    2t\\
    -2t-1\\
    1\\
    0
\end{pmatrix}
=
-\eta^t_{1,2}
\begin{pmatrix}
    0\\
    1\\
    -1\\
    0
\end{pmatrix}
+\nu^t_{1,1}
\begin{pmatrix}
    1\\
    -1\\
    0\\
    0
\end{pmatrix}
\]
with the unique multipliers $\eta^t_{1,2}=1$ and $\nu^t_{1,1}=2t$.
Note, that $x^t$ is in fact a nondegenerate minimizer of $\text{D}(t)$. Clearly, DISJ-ND1 and DISJ-ND2 are fulfilled. Let us check DISJ-ND3. We have for its Hessian
\[
D^2L^{\text{D}(t)}\xt=\begin{pmatrix}
    2&0&0&0\\
    0&-2&0&0\\
    0&0&2&0\\
    0&0&0&2
\end{pmatrix}.
\]
For its tangent space we have
\[
\mathcal{T}_{x^t}^{\text{D}(t)}=\left\{\xi \in \R^4\,\left\vert\,\xi_1=\xi_2=\xi_3\right.\right\}.
\]
Hence for $\xi \in \mathcal{T}_{x^t}^{\text{D}(t)}$ with $\xi\ne 0$ it holds
\[
\xi^T D^2L^{\text{D}(t)}\xi=2\left(\xi_1^2-\xi_2^2+\xi_3^2+\xi_4^2\right)>0.
\]
Moreover, for the C-index of $x^t$ we also get $\mbox{MPCC-}CI=0$, hence, it is indeed a minimizer.

Let us consider the underlying MPCC:
\[
\begin{array}{rrl}
\mbox{MPCC}:&\min\limits_{x}& x_1^2-x_2^2-x_2+x_3^2+x_3+(x_4-1)^2\\
&\mbox{s.\,t.} &(x_1-x_2)\cdot (x_1+x_2)= 0, x_1-x_2\ge 0, x_1+x_2\ge 0,\\
&&(x_2-x_3)\cdot x_4= 0, x_2-x_3\ge 0, x_4\ge 0,
\end{array}
\]
together with $\bar x=(0,0,0,1)$, i.e.~the limiting point of $x^t$ for $t \to 0$.
The latter is an S-stationary point of MPCC. Indeed, it holds
\[
\begin{pmatrix}
    0\\
    -1\\
    1\\
    0
\end{pmatrix}
=
\bar \sigma_{1,2}
\begin{pmatrix}
    0\\
    1\\
    -1\\
    0
\end{pmatrix}
+
\bar \varrho_{1,1}
\begin{pmatrix}
    1\\
    -1\\
    0\\
    0
\end{pmatrix}
+
\bar \varrho_{2,1}
\begin{pmatrix}
    1\\
    1\\
    0\\
    0
\end{pmatrix}
\]
with the unique multipliers
$\bar \sigma_{1,2}=-1$, $\bar \varrho_{1,1}=0$, and $\bar \varrho_{2,1}=0$. 
Note, that since $a_{00}\left(\bar x\right)=a^0_{00}\left(\bar x\right)=\left\{1\right\}$, condition (\ref{eq:conditionls2}) is fulfilled, hence, also (\ref{eq:conditionls}). However, neither (\ref{eq:conditionks}) nor MPCC-ND2 hold.
\qed
\end{example}

The following Example \ref{ex:ks} illustrates that condition (\ref{eq:conditionks}), and, hence, also (\ref{eq:conditionls}) and (\ref{eq:conditionls2}), are not sufficient to guarantee that a sequence of minimizers $x^t$ of $\text{D}(t)$ converges to a minimizer $\bar x$ of MPCC, cf. Theorem \ref{thm:s-statks}.

\begin{example}[Convergence under (\ref{eq:conditionks})]
\label{ex:ks}
We consider the following disjunctive regularization with $n=4$ and $\kappa=2$:
\[
\begin{array}{rrl}
\text{D}(t):&\min\limits_{x}& x_1^2+x_1x_3-x_2-x_3^2+(x_4-1)^2\\
&\mbox{s.\,t.} &\max\left\{t-(x_1-x_2),t-x_3\right\} \ge 0, x_1-x_2\ge 0, x_3\ge 0,\\
&&\max\left\{t-x_2,t-x_4\right\} \ge 0, x_2\ge 0, x_4\ge 0.
\end{array}
\]
The point $x^t=(t,t,0,1)$ is an S-stationary point of $\text{D}(t)$ for $t<\frac{1}{2}$ with
$\mathcal{H}_1\xt=\{2\}$, $\mathcal{H}_2\xt=\emptyset$, and $\mathcal{N}_1\xt=\mathcal{N}_2\xt=\{1\}$. Indeed, it holds
\[
\begin{pmatrix}
    2t\\
    -1\\
    t\\
    0
\end{pmatrix}
=
-\eta^t_{1,2}
\begin{pmatrix}
    0\\
    1\\
    0\\
    0
\end{pmatrix}
+\nu^t_{1,1}
\begin{pmatrix}
    1\\
    -1\\
    0\\
    0
\end{pmatrix}
+\nu^t_{2,1}
\begin{pmatrix}
    0\\
    0\\
    1\\
    0
\end{pmatrix}
\]
with the unique multipliers $\eta^t_{1,2}=1-2t$, $\nu^t_{1,1}=2t$, and $\nu^t_{2,1}=t$.
Note, that $x^t$ is in fact a nondegenerate minimizer of $\text{D}(t)$. Clearly, DISJ-ND1 and DISJ-ND2 are fulfilled. Let us check DISJ-ND3. We have for its Hessian
\[
D^2L^{\text{D}(t)}\xt=\begin{pmatrix}
    2&0&1&0\\
    0&0&0&0\\
    1&0&-2&0\\
    0&0&0&2
\end{pmatrix}.
\]
For its tangent space we have
\[
\mathcal{T}_{x^t}^{\text{D}(t)}=\left\{\xi \in \R^4\,\left\vert\,\xi_1=\xi_2=\xi_3=0\right.\right\}.
\]
Hence for $\xi \in \mathcal{T}_{x^t}^{\text{D}(t)}$ with $\xi\ne 0$ it holds
\[
\xi^T D^2L^{\text{D}(t)}\xi=2\left(\xi_1^2-\xi_3^2+\xi_4^2+\xi_1\xi_3\right)>0.
\]
Moreover, for the C-index of $x^t$ we also get $\mbox{DISJ-}CI=0$, hence, it is indeed a minimizer.

Let us consider the underlying MPCC:
\[
\begin{array}{rrl}
\mbox{MPCC}:&\min\limits_{x}& x_1^2+x_1x_3-x_2-x_3^2+(x_4-1)^2\\
&\mbox{s.\,t.} &(x_1-x_2)\cdot x_3= 0, x_1-x_2\ge 0, x_3\ge 0,\\
&&x_2\cdot x_4= 0, x_2\ge 0, x_4\ge 0,
\end{array}
\]
together with $\bar x=(0,0,0,1)$, i.e.~the limiting point of $x^t$ for $t \to 0$.
Note that MPCC-LICQ is fulfilled at $\bar x$. Moreover, condition (\ref{eq:conditionks}) is also fulfilled in view of $a_{00}\left(\bar x\right)=\{1\}$. According to Theorem \ref{thm:s-statks}, $\bar x$ is an S-stationary point of MPCC. Indeed, it holds
\[
\begin{pmatrix}
    0\\
    -1\\
    0\\
    0
\end{pmatrix}
=
\bar \sigma_{1,2}
\begin{pmatrix}
    0\\
    1\\
    0\\
    0
\end{pmatrix}
+\bar \varrho_{1,1}
\begin{pmatrix}
    1\\
    -1\\
    0\\
    0
\end{pmatrix}
+\bar \varrho_{2,1}
\begin{pmatrix}
    0\\
    0\\
    1\\
    0
\end{pmatrix}
\]
with the unique multipliers
$\bar \sigma_{1,2}=-1$, $\bar \varrho_{1,1}=0$, and $\bar \varrho_{2,1}=0$.
However, $\bar x$ is not a minimizer. In fact, it holds
$f\left(x_{\varepsilon}\right)<f\left(\bar x\right)$ for all  $x_{\varepsilon}=(0,0,\varepsilon,1)$ with $\varepsilon>0$.
Note that here the nondegeneracy assumptions from our Theorem \ref{thm:ttoc} are fulfilled, except of MPCC-ND2 at $\bar x$. 
\qed
\end{example}


Next, let us focus on the well-posedness results given in \cite{kanzow:Relax}. For this purpose, we define the following index sets for an S-stationary point $\bar x \in M$ of MPCC with multipliers $(\bar \sigma, \bar \varrho)$:
\[
a_{00}^1(\bar x)=\left\{
j \in a_{00}\left(\bar x\right) 
\,\left\vert \, \bar \varrho_{1,j}>0
\right. \right\}, \quad
a_{00}^2(\bar x)=\left\{
j \in a_{00}\left(\bar x\right) 
\,\left\vert \, \bar \varrho_{2,j}>0
\right. \right\}.
\]
Additionally, we define the set of critical directions:
\[
\mathcal{C}_{\bar x}=\left\{\xi \in \R^n \,\left\vert \,
\begin{array}{l}
\nabla^T F_{1,j}\left(\bar x\right)\xi=0, j \in a_{01}^-(\bar x) \cup a_{01}^+(\bar x)\cup a_{00}^1(\bar x),\\
\nabla^T F_{2,j}\left(\bar x\right)\xi=0, j \in a_{10}^-(\bar x) \cup a_{10}^+(\bar x)\cup a_{00}^2(\bar x)
\end{array}
\right.\right\}. 
\]
We now recall the MPCC-tailored second order condition first used in \cite{scholtes:2001}.

\begin{definition}[SSOSC, \cite{kanzow:Relax}]
   Let $\bar x \in M$ be an S-stationary point of MPCC with the multipliers $(\bar \sigma, \bar \varrho)$. We say that it fulfills the strong second order sufficiency condition (SSOSC) if for every $\xi \in {\mathcal{C}}_{\bar x} \backslash \{0\}$ it holds: 
    \[
    \xi^T D^2L\left(\bar x\right)\xi >0.
    \]
\end{definition}
Let us present Theorem 4.12 from \cite{kanzow:Relax} in a slightly modified form.

    

\begin{theorem}[Well-posedness, \cite{kanzow:Relax}]
\label{thm:wellposednessks}
   Let $\bar x \in M$ be an S-stationary point of MPCC that satisfies MPCC-LICQ and 
   SSOSC. Then, for all sufficiently small $t$ there exists at least one local minimizer $ x^t \in M^{\text{KS}(t)}$ of $\text{KS}(t)$ within a neighborhood of $\bar x$.
\end{theorem}

We elaborate on the relations between Theorems \ref{thm:wellposedness} and \ref{thm:wellposednessks}. Let a nondegenerate local minimizer $\bar x\in M$ of MPCC fulfill MPCC-ND4. Then, $\bar x$ is S-stationary and fulfills SSOSC. Both Theorems \ref{thm:wellposedness} and \ref{thm:wellposednessks} can be applied to get (unique) local minimizers of $\text{D}(t)$ and $\text{KS}(t)$, respectively, in the vicinity of $\bar x$. Note that Theorem \ref{thm:wellposedness} can be likely applied in the case of saddle points of MPCC, whereas Theorem \ref{thm:wellposednessks} does not need the conditions MPCC-ND2 or MPCC-ND4  to be imposed. The price to pay is that the assumptions of Theorem \ref{thm:wellposednessks} do not in general ensure nondegeneracy of the minimizers $ x^t\in M^{\text{KS}(t)}$ of $\text{KS}(t)$ even if the S-stationary point $\bar x\in M$ of MPCC is nondegenerate, see Example \ref{ex:ssosc}. According to the latter, the same conclusion is true if the nondegeneracy for $ x^t\in M^{\text{D}(t)}$ is checked with respect to $\text{D}(t)$. Overall, it becomes clear that both theorems are of independent interest. 
    
\begin{example}[Degeneracy under SSOSC]
\label{ex:ssosc}
    We consider the following MPCC with
$n=2$ and $\kappa=1$: 
\[
\begin{array}{rll}
\mbox{MPCC}:&\min\limits_x & x_1^2+x_2^2\\
&\mbox{s.t.}& x_1\cdot (1-x_2) = 0, x_1\ge 0, 1-x_2 \ge 0.
\end{array}
\]
In \cite{laemmel:anomalies} it was shown, that $\bar x=(0,0)$ is a nondegenerate S-stationary point violating MPCC-ND4.
Further, it is straightforward to see that $\bar x$ fulfills SSOSC.
Let us now focus on the corresponding Kanzow-Schwartz regularization
\[
\begin{array}{rll}
\text{KS}(t):&\min\limits_x & x_1^2+x_2^2\\
&\mbox{s.t.}&\Phi(x,t)\le0, x_1\ge 0, 1-x_2\ge 0,
\end{array}
\]
where 
\[
\Phi(x,t)=\left\{
\begin{array}{ll}
     (x_1-t)\cdot(1-x_2-t)&\text{for } x_1-t+1-x_2-t \ge 0, \\
     -\frac{1}{2}\left((x_1-t)^2+(1-x_2-t)^2\right)& \text{for } x_1-t+1-x_2-t < 0.
\end{array}
\right.
\]
We have that $ x^t=(0,0)$ is a Karush-Kuhn-Tucker point of $\text{KS}(t)$, since it holds
\[
\begin{pmatrix}
0\\
    0
\end{pmatrix}
=
 \mu^t_{1,1}
\begin{pmatrix}
    1\\
    0
\end{pmatrix}
\]
with the unique multiplier $ \mu^t_{1,1}=0$. Moreover, we have $ x^t \to \bar x$ for $t \to 0$. Although $\mathcal{H}_{12}\left( x\right)=\emptyset$ and, due to Theorem \ref{thm:LICQ-KS}, LICQ is fulfilled, the standard strict complementarity
is violated. Hence, $x^t$ is degenerate as a Karush-Kuhn-Tucker point for $\text{KS}(t)$. 

Instead, let us focus on the  disjunctive regularization 
\[
\begin{array}{rll}
\text{D}(t):&\min\limits_x & x_1^2+x_2^2\\
&\mbox{s.t.}&\max\left\{t-x_1,t-(1-x_2)\right\} \ge 0, x_1\ge 0, 1-x_2\ge 0.
\end{array}
\]
Clearly, $x^t=(0,0)$ is an S-stationary point of $\text{D}(t)$  with the
unique multiplier $\nu_{1,1}^t=0$, and $x^t \to \bar x$ for $t \to 0$.
However, $x^t$ violates DISJ-ND2 and is therefore degenerate for $\text{D}(t)$.
\qed
\end{example}

We conclude our examination of $\text{KS}(t)$ by considering its $\varepsilon$-stationary points. 
First, we adopt the definition of the latter notion as used in \cite{kanzow:Epsilon} to $\text{KS}(t)$.
\begin{definition}[$\varepsilon$-stationarity, \cite{kanzow:Epsilon}]
    We call $\widetilde x$ an $\varepsilon$-stationary point of $\text{KS}(t)$ for $\varepsilon >0$ if there exist multipliers
    \[
\widetilde\mu_{j}, \widetilde\mu_{1,j}, \widetilde\mu_{2,j},\,j\in\left\{1,\ldots,\kappa\right\},
\]
such that
\begin{equation}
    \label{eq:ekkt1}
   \begin{array}{rcl}
\displaystyle\left\|
\nabla f\left( \widetilde x\right) +\displaystyle\sum\limits_{j=1}^{\kappa}
\widetilde\mu_{j} \nabla \Phi_{j}\left(\widetilde x,t\right) -
\displaystyle\sum\limits_{j=1}^{\kappa}
\widetilde\mu_{1,j} \nabla F_{1,j}\left(\widetilde x\right)
-
\displaystyle\sum\limits_{j=1}^{\kappa}
\widetilde\mu_{2,j} \nabla F_{2,j}\left(\widetilde x\right)\right\|_{\infty}
&\le&\varepsilon,
\end{array}
\end{equation}

\begin{equation}
    \label{eq:ekkt2}
 \begin{array}{c}
 \Phi_{j}\left(\widetilde x,t\right)\le \varepsilon,
 F_{1,j}\left(\widetilde x\right) \ge -\varepsilon,
  F_{2,j}\left(\widetilde x\right) \ge -\varepsilon,
j=1,\ldots,\kappa,
    \end{array}
\end{equation}
\begin{equation}
    \label{eq:ekkt3}
 \begin{array}{c}
 \widetilde\mu_{j}\ge - \varepsilon,\widetilde\mu_{1,j}\ge - \varepsilon, \widetilde\mu_{2,j}\ge - \varepsilon, j=1,\ldots,\kappa,
     \end{array}
\end{equation}
\begin{equation}
    \label{eq:ekkt4}
 \begin{array}{c}
|\widetilde\mu_{j}\cdot \Phi_{j}\left(\widetilde x,t\right)|\le\varepsilon,  |\widetilde\mu_{1,j}\cdot F_{1,j}\left(\widetilde x\right)|\le\varepsilon, 
 |\widetilde\mu_{2,j}\cdot F_{2,j}\left(\widetilde x\right)|\le\varepsilon,
 j=1,\ldots,\kappa.
    \end{array}
\end{equation}
\end{definition}

In \cite{kanzow:Epsilon}, the authors suggest to compute $\varepsilon$-stationary points of $\text{KS}(t)$ instead of its Karush-Kuhn-Tucker points, and then to drive the parameter $t$ to zero. Let us recall their result on the convergence of $\varepsilon$-stationary points of $\text{KS}(t)$ towards C-stationary points of MPCC.  

\begin{theorem}[\cite{kanzow:Epsilon}]
\label{thm:ekkt}
Let $t^k\downarrow 0$, $\varepsilon^k=o\left(t_k\right)$.
  Suppose a sequence of $\varepsilon^k$-stationary points $\widetilde x^{k}$ of $\text{KS}(t)$ with multipliers $\left(\widetilde\mu^{k}, \widetilde\mu_{1}^{k}, \widetilde\mu_{2}^{k}\right)$ converges to
  $\bar x$ for $k \to \infty$.
  Let MPCC-LICQ be fulfilled at $\bar x \in M$.
  Further, assume that there is a constant $c>0$ such that
 for all $j\in a_{00}\left(\bar x\right)$ and all $k$ sufficiently large, it holds:
\begin{equation}
    \label{eq:condksepsilon}
\left(F_{1,j}\left(\widetilde x^k\right),F_{2,j}\left(\widetilde x^k\right)\right)\notin [(t_k,(1+c)t_k)\times((1-c)t_k,t_k)]\cup
 [((1-c)t_k,t_k)\times(t_k,(1+c)t_k)]. 
\end{equation}
Then, $\bar x$ is an C-stationary point of MPCC.
\end{theorem}

Unlike in Theorem \ref{thm:t-sequence}, C-stationarity of the limiting point cannot be ensured in Theorem \ref{thm:ekkt} without an additional assumption (\ref{eq:condksepsilon}) on the approximating points. For guaranteeing convergence to M- or S-stationary points even more demanding assumptions had to be stated in \cite{kanzow:Epsilon}.
Their necessity can be explained by the fact that $\varepsilon$-stationary points may not approximate Karush-Kuhn-Tucker but rather Fritz-John points of $\text{KS}(t)$ if $\varepsilon \rightarrow 0$, see Example \ref{ex:fjp}. Here, we again encounter the phenomena of instability, which is due to the intrinsic degeneracy of $\mbox{KS}(t)$.

\begin{example}[Approximation of Fritz-John points]
\label{ex:fjp}
We consider the MPCC given in Example 4 of \cite{kanzow:Epsilon} with
$n=2$ and $\kappa=1$:
\[
\begin{array}{rll}
MPCC:&\min\limits_x & -x_1-x_2\\
&\mbox{s.t.}& x_1\cdot x_2 = 0, x_1\ge 0, x_2 \ge 0,
\end{array}
\]
where $t>0, \varepsilon^t=t^2$.
As given in \cite{kanzow:Epsilon}, the points $\widetilde x^t=\left(t-\varepsilon^t,t-\varepsilon^t\right)$ are $\varepsilon^t$-stationary points of $\text{KS}(t)$ with multipliers $\widetilde\mu^{t}=\frac{1}{\varepsilon^t}, \widetilde\mu_{1,1}^{t}=0, \widetilde\mu_{2,1}^{t}=0$.
Note that the Kanzow-Schwartz regularization  of the considered MPCC, does not have any Karush-Kuhn-Tucker points. However, it has a Fritz-John point $\widehat x^t=(t,t)$, since it holds for $\mbox{KS}(t)$:
\[
\widehat \gamma^t
\begin{pmatrix}
    -1\\-1
\end{pmatrix}
=-\widehat \mu^t \begin{pmatrix}
    0 \\ 0
\end{pmatrix}
+ \widehat \mu^t_{1,1}\begin{pmatrix}1\\0
\end{pmatrix}
+ \widehat \mu^t_{2,1}\begin{pmatrix}0\\1
\end{pmatrix}
\]
with multipliers
$\widehat \gamma^t=\widehat \mu^t_{1,1}= \widehat \mu^t_{2,1}=0$ and $\widehat \mu^t > 0$.    
Thus, $\widetilde x^t$'s do approximate a Fritz-John point of $\mbox{KS}(t)$ for $\varepsilon^t \rightarrow 0$. Note that $\varepsilon^t$-stationary points $\widetilde x^t$ converge for $t \to 0$ to the nondegenerate C-stationary point $(0,0)$ of MPCC with C-index equal to one. \qed
\end{example}

Further, we emphasize that the consideration of $\varepsilon$-stationary points instead of Karush-Kuhn-Tucker points may cause catastrophic cancellation of multipliers, see Example \ref{ex:catcanc}. We suggest that this observation helps to understand the numerical issues observed in e.g.~\cite{hoheisel:2013, nurkanovic:2024}, when MPCCs are solved by means of the Kanzow-Schwartz regularization $\mbox{KS}(t)$. In \cite{hoheisel:2013}, it is namely concluded as follows: ''Despite some theoretical advantages of some of the newer relaxation schemes, the numerical comparison favours the oldest relaxation scheme due to Scholtes.''


\begin{example}[Catastrophic cancellation]
\label{ex:catcanc}
We consider the MPCC from Example \ref{ex:fjp}
as well as the $\varepsilon^t$-stationary points $\widetilde x^t=\left(t-\varepsilon^t,t-\varepsilon^t\right)$ of the corresponding $\text{KS}(t)$ with $\varepsilon^t=t^2$. 
For any multipliers $\left(\widetilde\mu^{t}, \widetilde\mu_{1,1}^{t}, \widetilde\mu_{2,1}^{t}\right)$, we have due to (\ref{eq:ekkt1}):
\[
\lim\limits_{t\to 0} -1-\widetilde\mu^{t}\left(\widetilde x_1^t-t\right)-
\widetilde\mu_{1,1}^{t}=0.
\]
Additionally, we have in view of (\ref{eq:ekkt4})
that $\widetilde\mu_{1,1}^{t}\to 0$ for $t\to 0$.
Hence,
\[
\lim\limits_{t\to 0} -\widetilde\mu^{t}\left(\widetilde x_1^t-t\right)=1.
\]
Since $\widetilde x_1^t-t=(t-\varepsilon^t)-t$ approaches zero for $t\to 0$ we have that $\widetilde\mu^{t}$ must be unbounded.
However, $x_1^t-t$ is a difference of two small numbers. 
Therefore, this calculation might cause catastrophic cancellation for small $t$, which is amplified by the multiplication with unbounded $\widetilde\mu^{t}$.
Moreover, the same numerical issue occurs with respect to the second component.
We point out that the catastrophic cancellation generates from the degeneracies in $\mbox{KS}(t)$, cf. Example \ref{ex:fjp}. 
\qed
\end{example}

We proceed our comparison with a continuous quadrant penalty formulation from 
\cite{cafieri:2023}. There, the authors consider logical constraints, i.e.~constraints of the form
\[u(x)\le 0 \quad \mbox{or}\quad v(x)\ge0,\]
where the defining functions $u:\R^n\to\R$ and $v:\R^n\to\R$ are real-valued.
For this type of constraint, a continuous optimization approach was proposed in \cite{cafieri:2023} as an alternative to the usual discrete optimization formulations. In particular, a family of smooth quadratic penalty functions was constructed. With $\beta>1$ they are defined as 
    \[
g_{\beta} \left(u,v\right)=\left\{
\begin{array}{ll}
    0 & \mbox{for } u\le 0 \text{ or }v\ge 0, \\
     u^2 & \mbox{for } 0 < u\le\frac{-v}{\beta},\\
          \frac{1}{1-\beta^2}\cdot \left(u^2-2\beta u v+v^2\right) & \mbox{for } \frac{-v}{\beta} < u\le -\beta v,\\
          v^2 & \mbox{for } \frac{-u}{\beta}\le v <0.
\end{array}\right.
\]
In what follows, we use $g_\beta$ to get another functional description of the feasible set $M^{\text{D}(t)}$.
For this, recall that the constraints
$\max\{t-F_{1,j}(x),t-F_{2,j}(x)\}\ge0$
can be equivalently written as
\[
F_{1,j}(x)-t\le 0\quad \mbox{or}\quad
t-F_{2,j}(x)\ge 0.
\]
Let us consider the following regularization of MPCC:
\[
\text{QPF}(t): \quad
\min_{x} \,\, f(x)\quad \mbox{s.\,t.} \quad x \in M^{\text{QPF}(t)}
\]
with the feasible set
\[
    M^{\text{QPF}(t)}=\left\{
    x \in\R^n\, \left\vert\,   
    \begin{array}{l} 
    g_{\beta} \left(F_{1,j}(x)-t,t-F_{2,j}(x)\right)= 0,\\ 
    F_{1,j}(x) \ge 0, F_{2,j}(x)\ge 0, j=1,\ldots,\kappa 
            \end{array}
    \right. \right\}.
\]
It is straightforward to see that the feasible sets of $\text{QPF}(t)$ and $\text{D}(t)$ coincide, i.e.~$M^{\text{QPF}(t)}=M^{\text{D}(t)}$.
However, considering $\text{QPF}(t)$ as an instance of NLP will again result in a violation of LICQ.
 More precisely, if $x \in M^{\text{QPF}(t)}$ is a 
feasible point of $\text{QPF}(t)$ with $\mathcal{H}_{12}\left(x\right)\ne \emptyset$, we have
   \[
    \nabla g_{\beta} \left(F_{1,j}\left( x\right)-t,t-F_{2,j}\left( x\right)\right)=0,
    j \in \mathcal{H}_{12}\left( x\right).
    \]
If such $x$ turns out to be additionally a Karush-Kuhn-Tucker point of $\text{QPF}(t)$, then it must thus be degenerate. Once again, the study of $\text{QPF}(t)$ within the scope of Morse theory is hampered. E.g., it becomes impossible to trace the quadratic index of Karush-Kuhn-Tucker points of $\text{QPF}(t)$.

\color{black}

\section{Numerical results}
\label{sec:num}

To conclude our studies, we provide a preliminary numerical test on the disjunctive regularization $\text{D}(t)$.
We implemented 75 small- and middle-sized instances from the MacMPEC database of \cite{macmpec:2015} as well as Examples \ref{ex:mintosaddle}--\ref{ex:fjp} from this paper in Python (version 3.13.9). We used the SLSQP method from SciPy (version 1.16.3) of \cite{scipy:2020} for solving them. In addition, we repeated the test with IPOPT (version 3.14.19) of \cite{ipopt:2006} using the python wrapper cyipopt (version 1.6.1) of \cite{cyipopt:2026}.
In order to assess the feasibility of an obtained solution, we used the maximum constraint violation from \cite{hoheisel:2013}.
 Given additional equality constraints $h(x)=0$ and inequality constraints $g(x)\ge0$ in MPCC, the maximum constraint violation is defined as
\[
\text{maxvio}(x) = \max\left\{-\min\left\{0, g(x)\right\}, \left|h(x)\right|, \left|\min\left\{F_{1}(x), F_{2}(x)\right\}\right|\right\}.
\]
For the implemented algorithm we again  closely follow \cite{hoheisel:2013}.  
Let $R\in \{\text{KS},\text{D},\text{S}\}$ be a particular regularization tested by Algorithm \ref{alg:ks}.

{
\centering
\begin{minipage}{1\linewidth}
\begin{algorithm}[H]
\caption{}
\label{alg:ks}
\small
\begin{algorithmic}
\algblock{Begin}{End}
\Require A starting vector $x^0$, an initial relaxation parameter $t_0$, and parameters $t_{\text{min}} > 0$,
$\epsilon > 0$, $\sigma \in (0, 1)$.
\Begin
\State $k\gets0.$
\While{$(t_k \ge t_{\text{min}}$ and $\text{maxvio}\left(x^k\right) > \epsilon)$ or $k = 0$}
\State Find an approximate solution $x^{k+1}$ of the relaxed problem $R(t_k)$ using $x^k$ as \State starting vector.
\State $t_{k+1} \gets\max\limits_{l=1,2,3,\ldots}\left\{\sigma^l\cdot t_k\,\left\vert\,x^{k+1}\text{ is not feasible for } R\left(\sigma^l\cdot t_k\right)\text{ and }\sigma^l\cdot t_k\ge t_{\text{min}}\right.\right\}$.
\State $k\gets k+1$.
\EndWhile
\End
\State \textbf{Return:} The final iterate $\bar x_R := x^k$, the corresponding function value $f\left(\bar x_R\right)$, the maximum constraint violation $\text{maxvio}\left(\bar x_R\right)$, and the computational time.
\end{algorithmic}
\end{algorithm}
\end{minipage}
}

\noindent
The parameters were also chosen as in \cite{hoheisel:2013}: 
 \[
 t_0=1,  \quad 
 t_{\text{min}}=10^{-15}, \quad \epsilon=10^{-6}, \quad \sigma=\left\{
 \begin{array}{ll}
      0.01& \text{for }R \in \{\text{KS},\text{D}\}, \\
      0.0001& \text{for } R = \text{S}.
 \end{array}\right.
 \]
 Finally, if the starting vector $x^0$ was not specified in the problem data, we set each non-declared value to zero. Note that therefore Example \ref{ex:biactive} differs from \verb|kth1| as given in the MacMPEC database with respect to the choice of $x^0$.
Instances, for which the returned maximum constraint violation $\text{maxvio}\left(\bar x_R\right)$ exceeded $\epsilon$, were considered infeasible, and the function value 
$f\left(\bar x_R\right)$ as well as the computational time were set to $\infty$, cf. \cite{hoheisel:2013}.
For assessing the obtained functional values $f\left(\bar x_R\right)$, we once again follow \cite{hoheisel:2013} to define the normalized relative error. The latter is adjusted to the cases when
$\min\left\{f\left(\bar x_R\right)\,\left\vert\,R \in \{\text{KS},\text{D},\text{S}\}\right.\right\}\in\{0,\infty\}$:
\[
\bar f_R=\left\{
\begin{array}{ll}
     \displaystyle \frac{f\left(\bar x_R\right)}{\epsilon_M}&\text{for }
     \min\left\{f\left(\bar x_A\right)\,\left\vert\,A \in \text{KS},\text{D},\text{S}\}\right.\right\}=0,\\
          \displaystyle \infty&\text{for }
     \min\left\{f\left(\bar x_A\right)\,\left\vert\,A \in \text{KS},\text{D},\text{S}\}\right.\right\}=\infty,\\
     \displaystyle \frac{f\left(\bar x_R\right)-\min\left\{f\left(\bar x_A\right)\,\left\vert\,A \in \{\text{KS},\text{D},\text{S}\}\right.\right\}}{|\min\left\{f\left(\bar x_A\right)\,\left\vert\,A \in \{\text{KS},\text{D},\text{S}\}\right.\right\}|}&\text{else,}
\end{array}\right.
\]
where 
$\epsilon_M$ denotes the machine epsilon \verb|numpy.finfo(float).eps| $=2^{-52}$.
Analogously, we defined the relative computational time
\[
\bar \tau_R=\left\{
\begin{array}{ll}
     \displaystyle \frac{\tau_R}{\epsilon_M}&\text{for }
     \min\left\{\tau_A\,\left\vert\,A \in \text{KS},\text{D},\text{S}\}\right.\right\}=0,\\
          \displaystyle \infty&\text{for }
     \min\left\{\tau_A\,\left\vert\,A \in \text{KS},\text{D},\text{S}\}\right.\right\}=\infty,\\
     \displaystyle \frac{\tau_R-\min\left\{\tau_A\,\left\vert\,A \in \{\text{KS},\text{D},\text{S}\}\right.\right\}}{|\min\left\{\tau_A\,\left\vert\,A \in \{\text{KS},\text{D},\text{S}\}\right.\right\}|}&\text{else,}
\end{array}\right.
\]
where $\tau_R$ denotes the CPU-time needed for solving the regularization $R$ averaged over ten runs.
The results obtained are presented in performance profiles -- originally introduced in
\cite{dolan:2002}, see Figures \ref{fig:diaslsqp} and \ref{fig:diaipopt}. Diagrams were generated using Matplotlib Pyplot (version 3.10.8) of \cite{matplotlib:2007}.

\begin{figure}
\begin{center}
\subfigure{\includegraphics[width=0.40\textwidth]{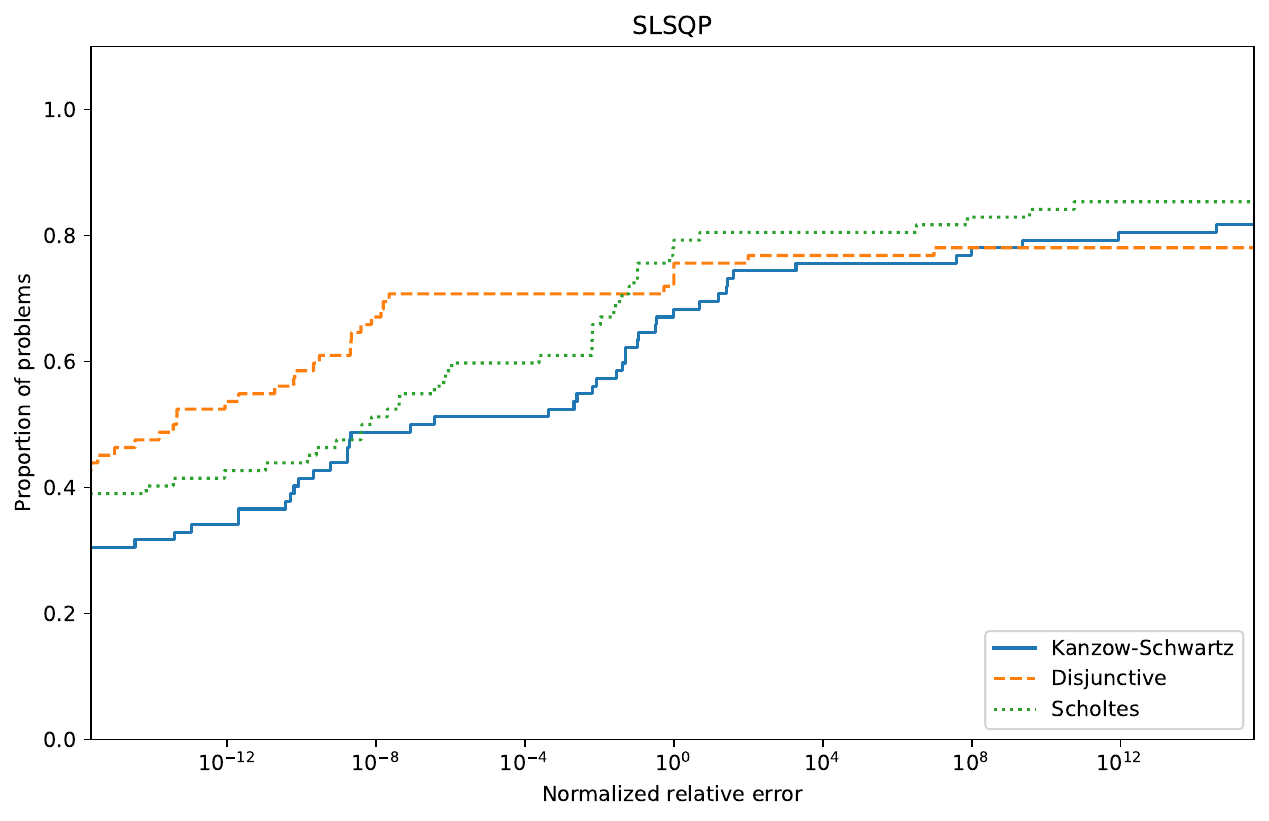}{}}
\subfigure{\includegraphics[width=0.4\textwidth]{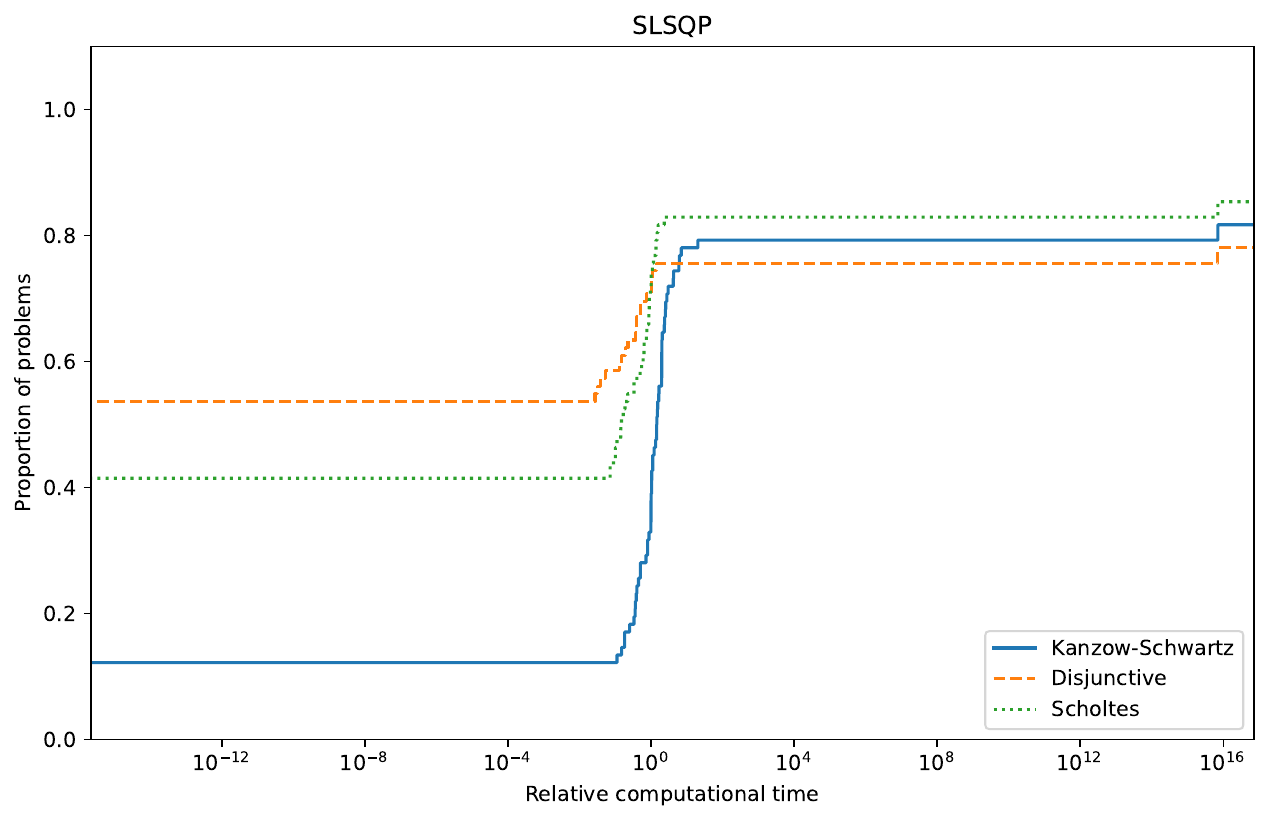}{}}\end{center}
\begin{center}
\subfigure{\includegraphics[width=0.4\textwidth]{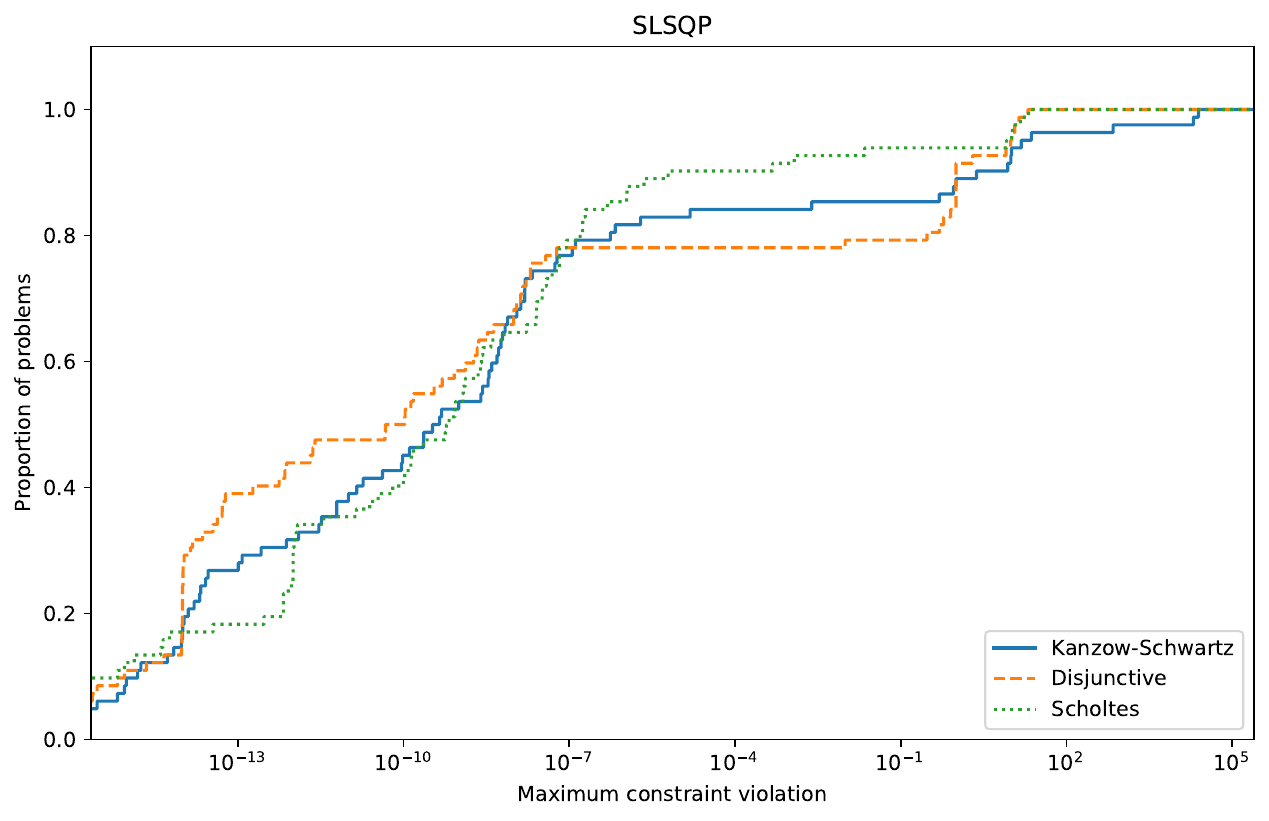}{}}
\end{center}
\caption{Performance profiles using the SLSQP solver}
\label{fig:diaslsqp}
\end{figure}

\begin{figure}
\begin{center}
\subfigure{\includegraphics[width=0.4\textwidth]{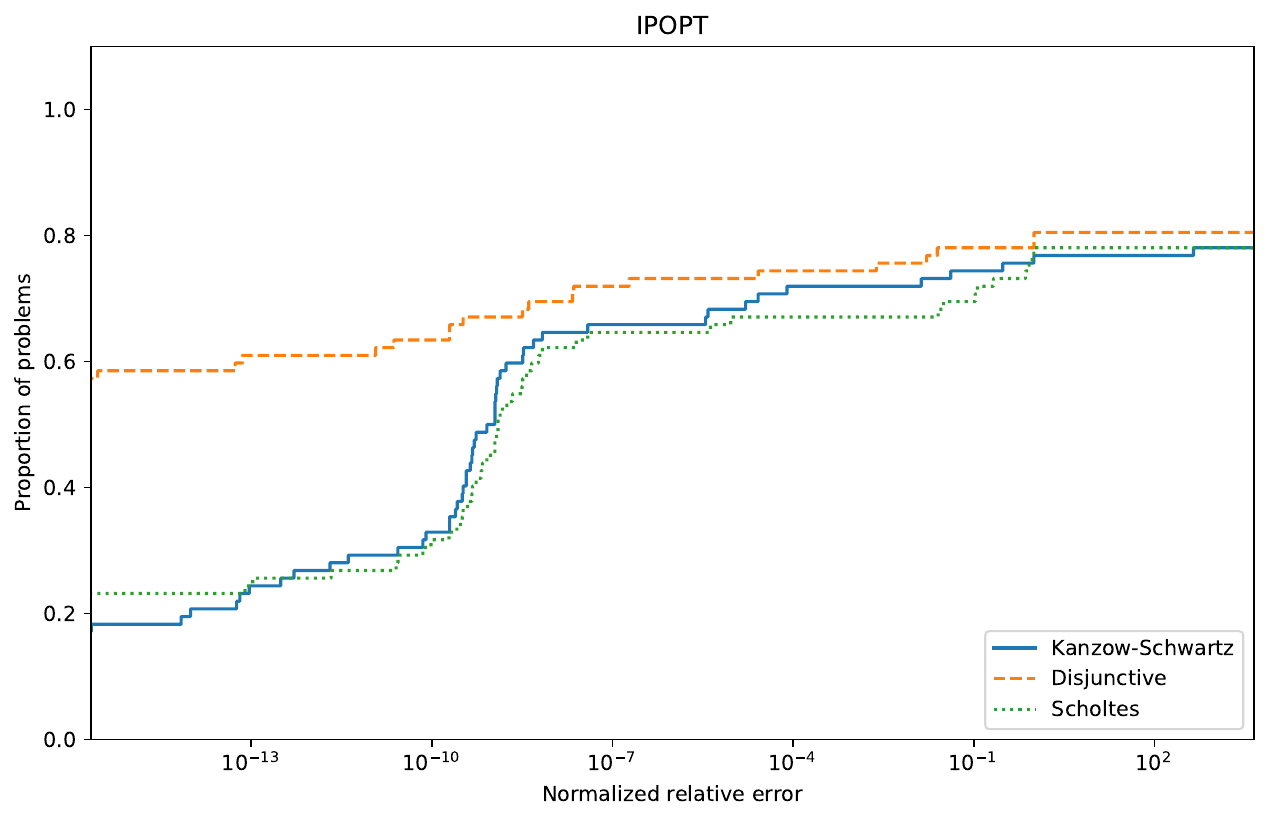}{}}
\subfigure{\includegraphics[width=0.4\textwidth]{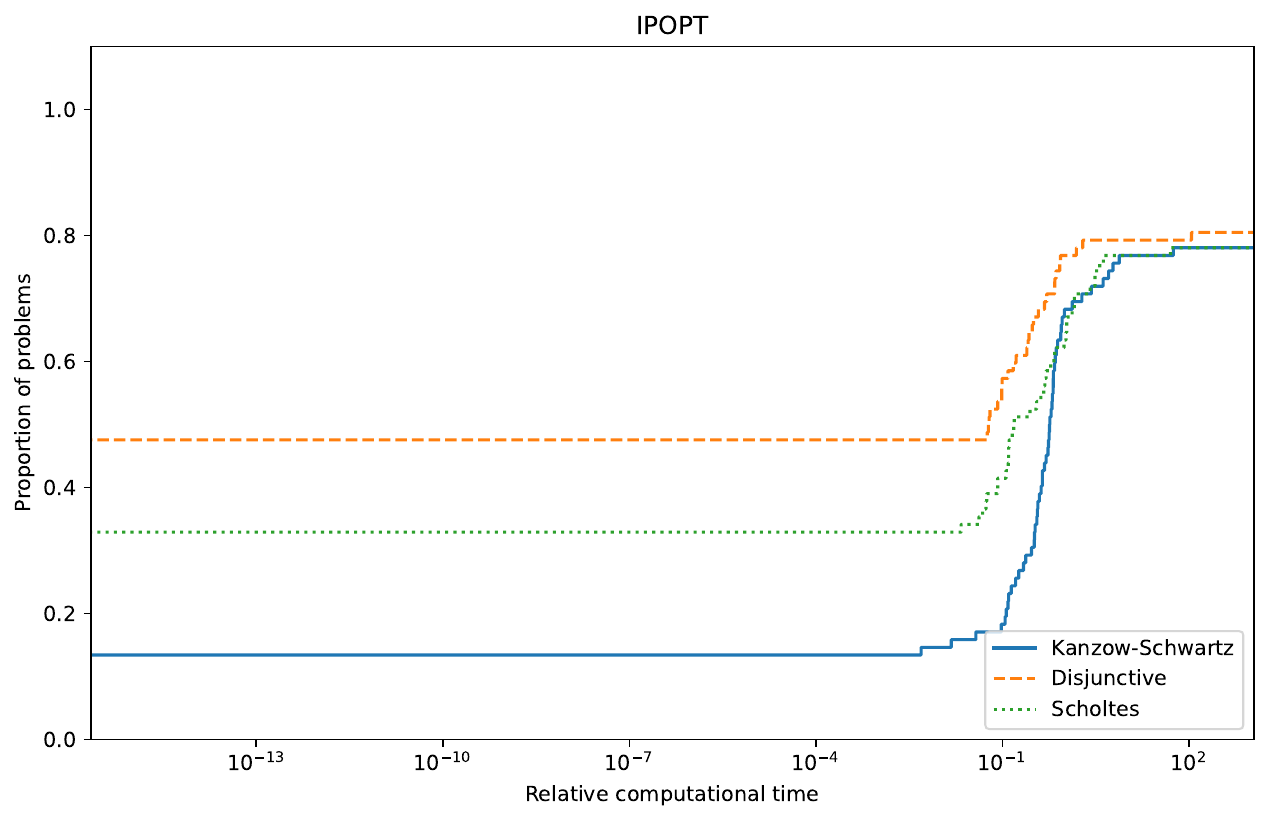}{}}\end{center}
\begin{center}
\subfigure{\includegraphics[width=0.4\textwidth]{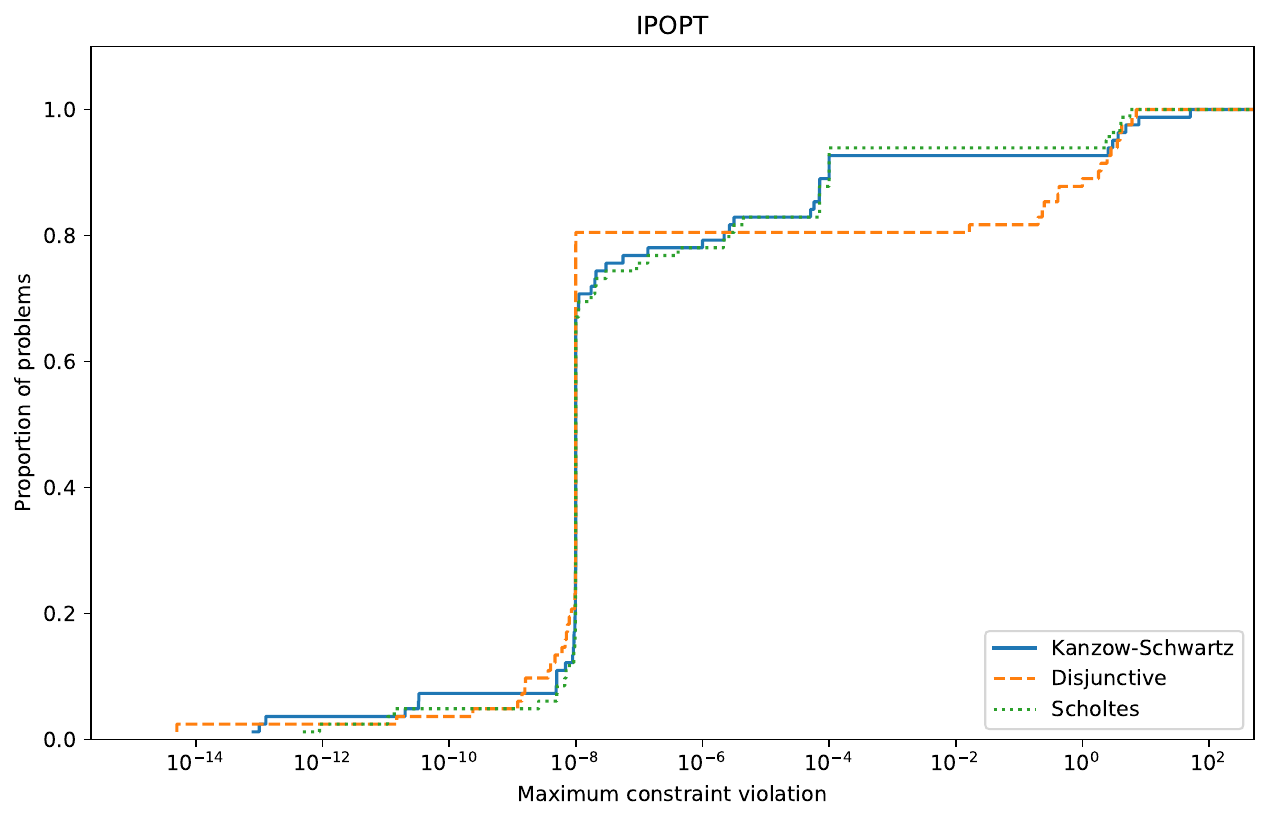}{}}
\end{center}
\caption{Performance profiles using the IPOPT solver}
\label{fig:diaipopt}
\end{figure}

For both the SLSQP and the IPOPT solvers, the performance profiles indicate that the disjunctive regularization $\text{D}(t)$ does best over the Kanzow-Schwartz regularization $\text{KS}(t)$ and the Scholtes regularization $\text{S}(t)$, at least if solving MPCCs with high accuracy for the latter.
In particular, if considering the IPOPT solver and the corresponding normalized relative errors, the disjunctive regularization outperforms the Scholtes regularization over the whole accuracy range. This is remarkable, since in \cite{nurkanovic:2024} the authors come to the conclusion that the Scholtes regularization solved by IPOPT was "the most successful method-solver combination". We see that the performance profiles in Figure \ref{fig:diaipopt} indicate that this statement has to be questioned, if not corrected.
For the comparison of computational time 
similar conclusions as for the normalized relative error can be made.

Exemplarily, we give for the problems \verb|ex9.2.2| (optimal value 100.0), \verb|ralph1| (optimal value 0.0), and \verb|scholtes4| (optimal value -3.07336e-7) the obtained functional values in Table \ref{tab:examples}. Those problems do not have S-stationary points and are considered to be difficult to solve, cf.~\cite{hoheisel:2013, raghunathan:2005}. Also here, the superiority of the disjunctive regularization is evident.

\begin{table}[H]
\centering
\begin{tabular}{|l|l|l|l|}
\hline
 & $\text{KS}(t)$ &$\text{D}(t)$  & $\text{S}(t)$ \\ \hline
\verb|ex9.2.2| (SLSQP)     & 110.5         & 100         & 100        \\ \hline
\verb|ex9.2.2| (IPOPT)      & infeasible        & infeasible        & infeasible       \\ \hline
\verb|ralph1| (SLSQP)     & infeasible        & -1.17684e-14         & infeasible        \\ \hline
\verb|ralph1| (IPOPT)      & infeasible        & -2.998e-08        & infeasible       \\ \hline
\verb|scholtes4| (SLSQP)     & -1.36158e-06         & -8.2601e-13         & infeasible        \\ \hline
\verb|scholtes4| (IPOPT)      & infeasible        & -3.07819e-09       & infeasible       \\ \hline
\end{tabular}
\caption{Functional values $f\left(\bar x_R\right)$ in comparison}
\label{tab:examples}
\end{table}

As the final part of our numerical tests, we  treat the disjunctive regularization $\text{D}(t)$ with a more tailored solver. For that, let us recall that $\text{D}(t)$ is a disjunctive optimization problem. Therefore, a solver designed exclusively for the latter type of optimization problems might yield better results than the NLP-solvers SLSQP and IPOPT. In order to examine this, we implemented the same 82 problems in GAMS (version: 52.5.0) and used the disjunctive optimization solver LogMIP of \cite{logmip} to solve $\text{D}(t)$. 
For simplicity, we set $t=t_{\text{min}}$ and solved the disjunctive regularization $\text{D}(t)$ by calling the LogMIP solver.
The obtained results were then compared with those of SLSQP and IPOPT as described in Algorithm \ref{alg:ks}, see  Figure \ref{fig:diadisj}.
While the performance profiles still indicate that IPOPT is best in terms of the normalized relative error, there is a clear tradeoff to the computational time. Further, the maximum constraint violation is noticeable. Namely, IPOPT seems to make use of the regularized feasible set best among the solvers, while LogMIP convinces if it comes to high accuracy. 

\begin{figure}[h]
\begin{center}
\subfigure{\includegraphics[width=0.4\textwidth]{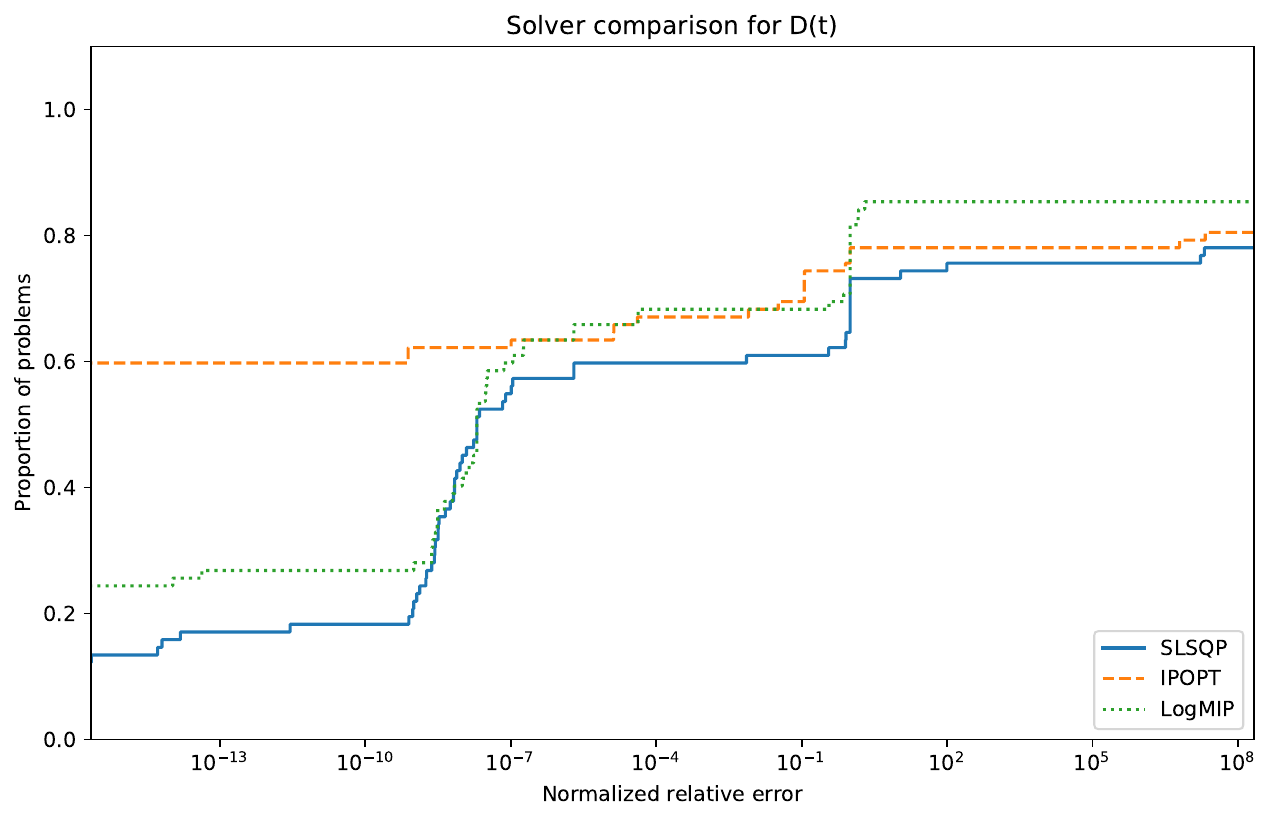}{}}
\subfigure{\includegraphics[width=0.4\textwidth]{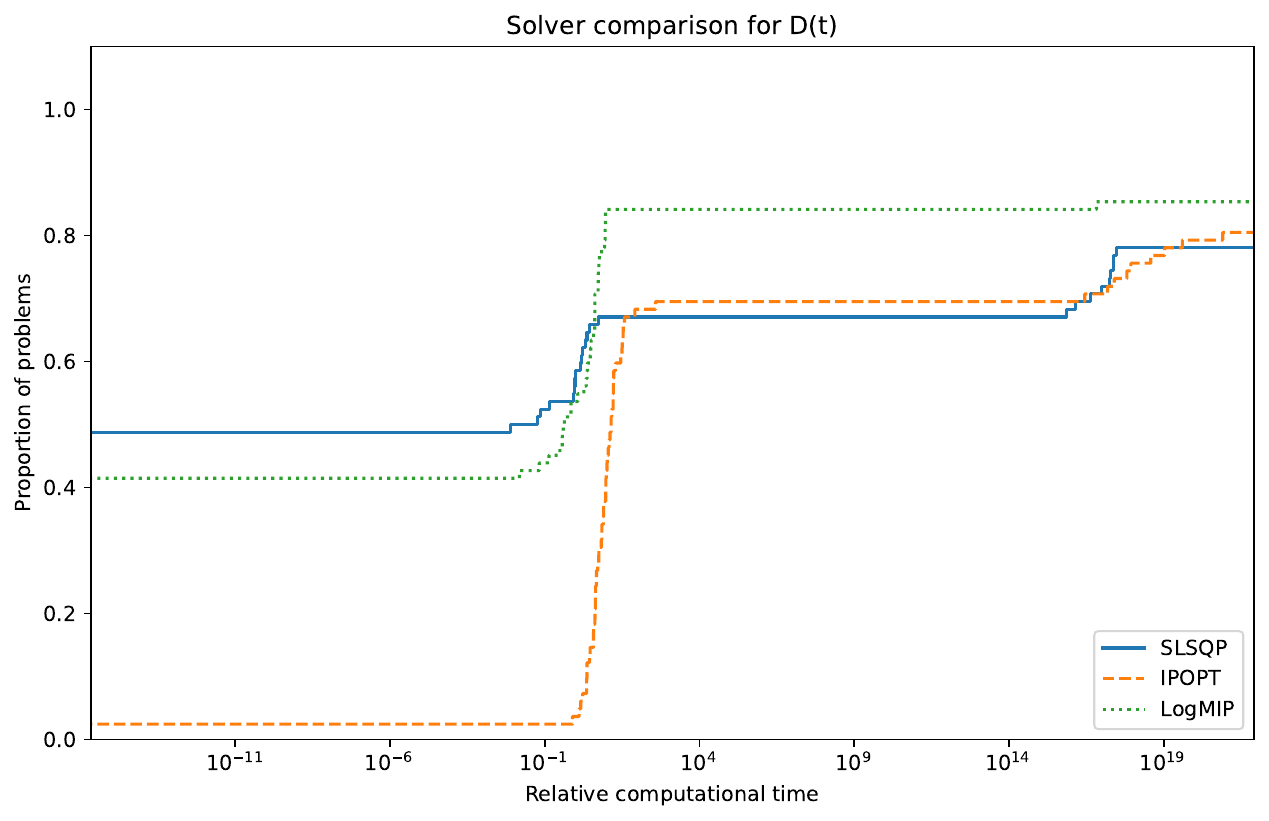}{}}    \end{center}
\begin{center}
\subfigure{\includegraphics[width=0.4\textwidth]{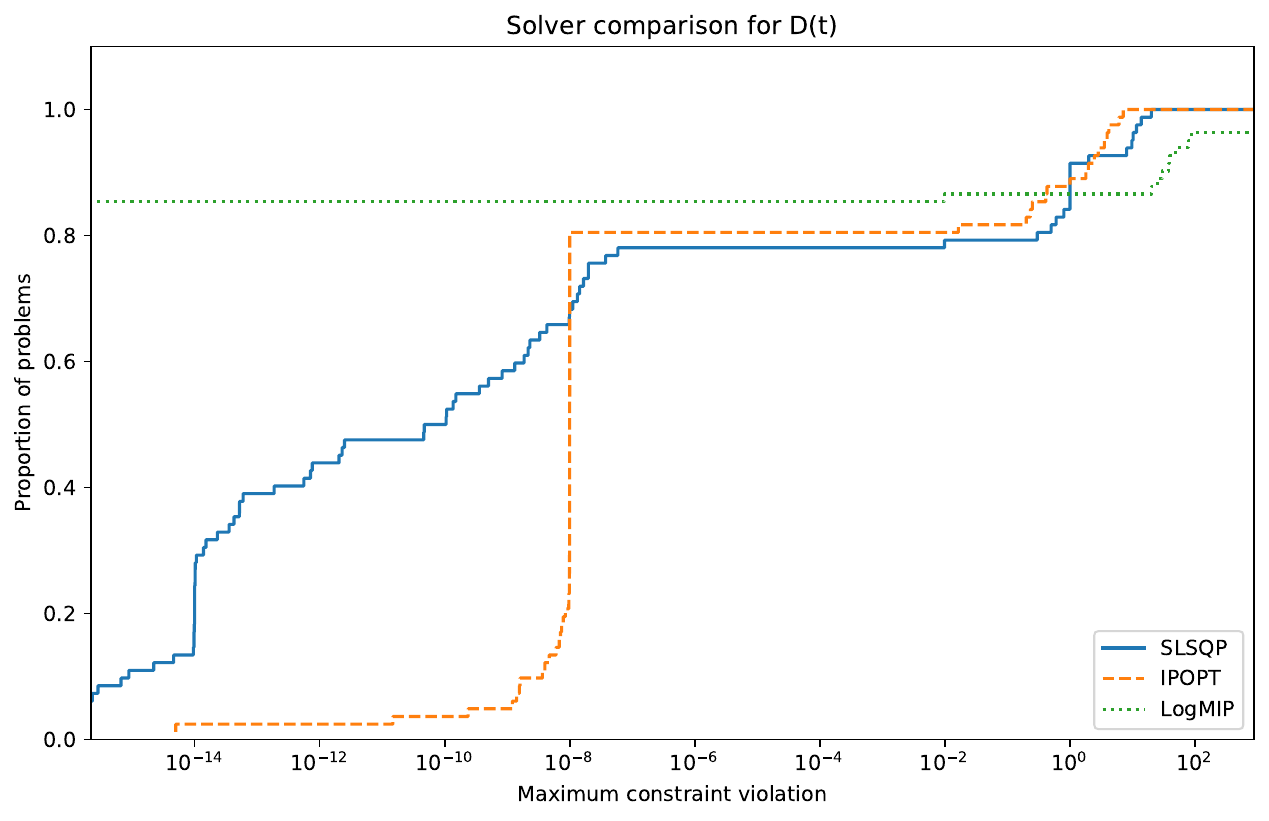}{}}
\end{center}
\caption{Performance profiles for $\text{D}(t)$ using different solvers}
\label{fig:diadisj}
\end{figure}


\bibliographystyle{apalike}
\bibliography{lit}
\end{document}